\documentclass[11pt] {article}% use larger type; default would be 10ptpts
\usepackage{amsmath,amssymb,amsfonts}
\usepackage[utf8]{inputenc} % set input encoding (not needed with XeLaTeX)
\usepackage{amsthm}

\usepackage{tikz}
\usepackage{hyperref}
\usepackage{graphicx}
\usepackage{amsfonts}

\usepackage{color}
\usepackage[english]{babel}

\usepackage{geometry} % to change the page dimensions
\usepackage{graphicx} 
\usepackage{array} % for better arrays (eg matrices) in maths
\usepackage{paralist} % very flexible & customisable lists (eg. enumerate/itemize, etc.)
\usepackage{verbatim} % adds environment for commenting out blocks of text & for better verbatim
\usepackage{subfig} % make it possible to include more than one captioned figure/table in a single float

\usepackage{textpos}

\newcommand{\R}{\mathbb{R}}
\newcommand{\Hy}{\mathbb{H}}
\newcommand{\C}{\mathbb{C}}

\newcommand{\D}{\mathbb{D}}

\newcommand{\n}{\vec{n}}
\newcommand{\s}{\mathbb{S}}
\newcommand{\Ar}{\mathring{A}}
\newcommand{\zb}{\bar{z}}

\tikzset { domaine/.style 2 args={domain=#1:#2} }

\newtheorem{theo}{Theorem}[section]
\newtheorem*{theo*}{Theorem}

\newtheorem{prop}[theo]{Proposition}
\newtheorem*{prop*}{Proposition}

\newtheorem{cor}{Corollary}[section]
\newtheorem*{cor*}{Corollary}
\newtheorem{de}{Definition }[section]

\newtheorem{remark}{Remark}[section]
\newtheorem{ex}{Example}[section]

\newcommand{\nocontentsline}[3]{}
\newcommand{\tocless}[2]{\bgroup\let\addcontentsline=\nocontentsline#1{#2}\egroup}

\title{Conformal Gauss Map Geometry and Application to Willmore Surfaces in Model Spaces}
\author{Nicolas Marque \thanks{Institut Mathématique de Jussieu, Paris VII, Bâtiment Sophie Germain, Case 7052, 75205 Paris Cedex 13, France. E-mail address : nicolas.marque@imj-prg.fr}}
\date{\today} % Activate to display a given date or no date (if empty),

\begin{document}

\maketitle
      % otherwise the current date is printed 
\abstract{ In this paper we make a detailed and self-contained study of the conformal Gauss map. Then, starting from the seminal work of R. Bryant \cite{bibdualitytheorem} and the notion of conformal Gauss map, we recover many fundamental properties of Willmore surfaces. We also get new results like some characterizations of minimal and constant mean curvature (CMC) surfaces in term of their conformal Gauss map behavior.}

\tableofcontents
\hspace{0.5cm}
\section{Introduction}

The following is primarily concerned with the study of the Moebius geometry of surfaces through the lense of the conformal Gauss map. 
This generalization of the osculating circles (see example \ref{ouonintroduitlaGaussconforme} for a proper definition) arose as a more relevant tool for conformal geometry than the classical Gauss map. Present as early as  1923 in G. Thomsen's works (see \cite{bibthomsen}), it proved a precious auxiliary in the understanding of Willmore surfaces.

Given a Riemann surface $\Sigma$ and an immersion $\Phi \, : \, \Sigma \rightarrow \R^3$ of first fundamental form $g$, of Gauss map $\n$, of mean curvature $H$ and tracefree second fundamental form $\Ar$, its Willmore energy is defined as $$W(\Phi) = \int_\Sigma H^2 dvol_g.$$  Willmore surfaces are critical points of the Willmore energy.  They satisfy the Willmore equation : 

$$\mathcal{W} (\Phi) := \Delta H + \big| \Ar \big|^2 H = 0.$$

The Willmore energy was already  under scrutiny in the XIXth century in the study of elastic plates, but to our knowledge W. Blaschke was the first to state (see \cite{MR0076373}) its invariance by conformal diffeomorphisms of $\R^3$  (which was later rediscovered by T. Willmore, see  \cite{bibwill}) and to study it in the context of conformal geometry. This invariance by conformal diffeomorphisms is key in studying Willmore surfaces. Indeed T. Rivière introduced conservation laws satisfied by Willmore immersions  (see  (7.15), (7.16) and (7.30c) in \cite{bibpcmi}) and the corresponding conserved quantities. Y. Bernard then showed  (see \cite{bibnoetherwill}) that these quantites were a consequence of the invariance of $W$. We will denote them $V_{\mathrm{tra}}, V_{\mathrm{rot}}, V_{\mathrm{dil}}$ and $V_{\mathrm{inv}}$, corresponding respectively to translations, rotations, dilations and inversions (see theorem \ref{thmV} for the precise definition). These conserved quantities take center stage in T. Rivière's proof of the regularity of Willmore surfaces (see \cite{bibanalysisaspects}). W. Blaschke also found out, and R. Bryant rediscovered in \cite{bibdualitytheorem}, that $\Phi (\Sigma)$ is a Willmore surface if and only if its conformal Gauss map $Y$ is a minimal branched immersion. In essence the conformal Gauss map is to Willmore surfaces what the Gauss map is to Constant Mean Curvature (CMC) surfaces.  

The exploitation of this link has proved fruitful numerous times. For instance R. Bryant introduced the holomorphic quartic $\mathcal{Q} = \langle Y_{zz}, Y_{zz} \rangle dz^4$ and showed in his seminal work \cite{bibdualitytheorem} that Willmore spheres were in fact inversions of minimal surfaces (see also Eschenburg's lecture notes \cite{bibeschenburg}). The resulting classification of Willmore spheres has far reaching consequences. A. Michelat and T. Rivière later extended it to branched Willmore spheres in \cite{michelatclassifi}. On  a somewhat different register F. Hélein used integrable systems on the conformal Gauss map to induce a  Weierstrass representation of Willmore immersions (see \cite{loopiloop} or  \cite{weierstrassWillmore} for a simplified look). From this he extracted a necessary condition for a Willmore immersion to be the conformal transform of a minimal immersion in $\R^3$, $\s^3$ or $\Hy^3$, see theorem 10 in \cite{loopiloop}. However due to the non-explicit nature of his Weierstrass data, what this condition exactly entails remains somewhat unclear. 
%For a look at the possibilities offered by integrable systems, even in higher codimensions the reader might look at \cite{DPW}.

Determining necessary and sufficient conditions for a surface to be the conformal transform of a minimal (or CMC) surface in one of the three models ($\R^3$, $\s^3$ or $\Hy^3$) is in fact another application of the notions surrounding the conformal Gauss map and Bryant's functional.  Several results offering an interesting panorama revolved around the notion of isothermic immersion. For instance we refer the reader to F. Burstall, F. Pedit and U. Pinkall's work in \cite{MR1955628}, while combining theorem 2.2 in B. Palmer's work \cite{bibpalmerconfgaussmap} (attributed to G. Thomsen) and theorem 4.4 in \cite{bobohle} (attributed to private communications from K. Voss) yields  the following theorem.
\begin{theo}
\label{lesresultatsquonavaitdeja}
Let $\Phi \, : \, \D \rightarrow \R^3$ be a smooth conformal immersion. We assume that $\Phi$ has no umbilic points.  $\Phi$ is the conformal transform of a CMC immersion in one of the three models if and only if $\mathcal{Q}$ is holomorphic and $\Phi$ is isothermic\footnote{see Definition \ref{defisothermicimmersions} for a precise definition}.
\end{theo}

Our aim will be threefold. First we intend to offer an organic, self-contained and comprehensive view of the notions orbiting around the conformal Gauss map while formulating them for immersions in the three studied models : $\R^3$, $\s^3$ and $\Hy^3$. Section \ref{section1} and \ref{section2} will be devoted to this endeavour.  Our study will yield two notable results. First is a description of the action on the model spaces of elements in $SO(4,1)$ through conformal diffeomorphisms, as shown by the following proposition.

\begin{prop}
$SO(4,1)$ acts transitively through conformal diffeomorphisms on $\R^3 \cup \{ \infty \}$ and $\s^3$. More precisely : 
\begin{itemize}
\item
  Let $M\in SO(4,1)$ and $X \in \s^3 \subset \R^4$. Then the action of $M$ on $X$ is given by :   
$$  M.X =  \frac{V_\circ}{V_5}  $$ where $$V = M \begin{pmatrix} X \\ 1 \end{pmatrix} = \begin{pmatrix} V_\circ \\ V_5 \end{pmatrix}.$$
\item
  Let $M\in SO(4,1)$ and $x \in \R^3$. Then the action of $M$ on $x$  is given by :
$$    M.x = \frac{y_\diamond}{y_5 -y_4} $$ where $$y =  M \begin{pmatrix} x \\ \frac{|x|^2-1}{2} \\ \frac{|x|^2+1}{2} \end{pmatrix}  = \begin{pmatrix} y_\diamond \\ y_4 \\y_5 \end{pmatrix}.$$ 
%\item 
%$SO(4,1)$ acts transitively through conformal diffeomorphisms on $\Hy^3$  : 
%$$  \text{ for } M \in SO(4,1) \quad M.Z = \frac{\left( M \begin{pmatrix} Z_{123} \\ 1 \\Z_4\end{pmatrix} \right)_{1234} }{ \left( M  \begin{pmatrix} Z_{123} \\ 1 \\Z_4\end{pmatrix} \right)_4}.$$
\end{itemize}
\end{prop}
Second goal of the paper is a geometric characterization of the conformal Gauss map for conformally CMC immersions. More precisely, we say that  $\Phi \, : \, \D \rightarrow \R^3$ (respectively $X \, : \, \D \rightarrow \s^3$, $Z \, : \, \D \rightarrow \Hy^3$) is conformally CMC (respectively minimal) if and only if there exists a conformal diffeomorphism $\varphi$ of $\R^3 \cup \{ \infty \}$ (respectively $\s^3$, $\Hy^3$) such that $\varphi \circ \Phi$ (respectively $\varphi \circ X$, $\varphi \circ Z$) has constant mean curvature (respectively is minimal) in $\R^3$ (respectively $\s^3$, $\Hy^3$). We have the following theorem.
\begin{theo}
\label{lepetitlabelquisertarien}
Let $\Phi$ be a smooth conformal immersion from $\D$ to $\R^3$, and $X$ (respectively $Z$) its representation in $\s^3$ (respectively $\mathbb{H}^3$) through $\pi$ (respectively $\tilde \pi$)\footnote{see subsection \ref{representationsinthethreemodels} for the precise assessment of the representations in the three models, and subsection \ref{localconformalequivalences} for the definition of the projections $\pi$ and $\tilde \pi$. }. Let $Y$ be its conformal Gauss map. We assume the  set of umbilic points of $\Phi$ (or equivalently, see (\ref{projectionstereo5}) and (\ref{projectionhyper5}), $X$ or $Z$)  to be nowhere dense. Then 
\begin{itemize}
\item
$\Phi$ is conformally CMC (respectively minimal) in $\R^3$ if and only if $Y$ lies in an affine (respectively linear) hyperplane of $\R^{4,1}$ with lightlike normal.
\item 
$X$  is conformally CMC (respectively minimal) in $\s^3$ if and only if $Y$ lies in an affine (respectively linear) hyperplane of $\R^{4,1}$ with timelike normal.
\item
$Z$ is conformally CMC (respectively minimal) in $\Hy^3$ if and only if $Y$ lies in an affine (respectively linear) hyperplane of $\R^{4,1}$ with spacelike normal.
\end{itemize}
\end{theo}
Parts of theorem \ref{lepetitlabelquisertarien} (concerning immersions in $\R^3$) can be found in \cite{bibpalmerconfgaussmap} or in \cite{dorfwangtarxiv1} in arbitrary codimension. 
A notable part of section 2 and 3, dedicated to those results, will be based on J-H. Eschenburg's and B. Palmer's previous surveys (respectively  \cite{bibeschenburg} and\cite{bibpalmerconfgaussmap}).\\

We will then address how one can further study Willmore immersions through conformal maps. We will show that the conserved quantities $V_{\mathrm{tra}}$, $V_{\mathrm{rot}}$, $V_{\mathrm{dil}}$ and $V_{\mathrm{inv}}$  can be read on a matrix based on the conformal Gauss map and its invariances, thanks to the following theorem.
\begin{theo}
Let $\Phi \, : \, \D \rightarrow \R^3$ be a Willmore immersion, conformal,  of conformal Gauss map $Y$. 
Let 
$$\begin{aligned} 
\mu = \begin{pmatrix}\nabla Y_i Y_j - Y_i \nabla Y_j \end{pmatrix} = \nabla Y Y^T - Y \nabla Y^T. \end{aligned}$$
Then 
\sbox0{$\begin{matrix}0 & - \tilde V_{\mathrm{rot} \,3 } &  \tilde V_{\mathrm{rot} \, 2}\\\tilde V_{\mathrm{rot} \, 3} & 0 & - \tilde V_{\mathrm{rot} \, 1}\\- \tilde V_{\mathrm{rot} \, 2} & \tilde V_{\mathrm{rot} \, 1} & 0\end{matrix}$}

$$2 \mu = \begin{pmatrix} U &-\frac{V_{\mathrm{tra}} - V_{\mathrm{inv}}}{2} & \frac{ V_{\mathrm{tra}} + V_{\mathrm{inv}} }{2} \\ \left(\frac{V_{\mathrm{tra}} - V_{\mathrm{inv}}}{2} \right)^T & 0 & V_{\mathrm{dil}} \\ -\left(\frac{V_{\mathrm{inv}}+ V_{\mathrm{tra}} }{2} \right)^T & -V_{\mathrm{dil}} & 0 \end{pmatrix}$$
%$$2 \mu=\left(
%\begin{array}{ccc}
%U &\makebox[\wd0]{\large $ -\frac{V_{\mathrm{tra}} - V_{\mathrm{inv}}}{2}$}& \makebox[\wd0]{ \large $\frac{ V_{\mathrm{tra}} + V_{\mathrm{inv}} }{2}$}\\
%  \vphantom{\usebox{0}}\makebox[\wd0]{\large $\left(\frac{V_{\mathrm{tra}} - V_{\mathrm{inv}}}{2} \right)^T$}&\makebox[\wd0]{0}&\makebox[\wd0]{$V_{\mathrm{dil}}$} \\
%  \vphantom{\usebox{0}}\makebox[\wd0]{\large $-\left(\frac{V_{\mathrm{inv}}+ V_{\mathrm{tra}} }{2} \right)^T$}&\makebox[\wd0]{$-V_{\mathrm{dil}}$}&\makebox[\wd0]{0}
%\end{array}
%\right)$$
where $V_{\mathrm{tra}}, V_{\mathrm{dil}}, V_{\mathrm{rot}} $ and $V_{\mathrm{inv}}$ are defined in theorem \ref{equationsdivergencewillmore} and $$ U = \left(\begin{matrix}0 & - \tilde V_{\mathrm{rot} \,3 } &  \tilde V_{\mathrm{rot} \, 2}\\\tilde V_{\mathrm{rot} \, 3} & 0 & - \tilde V_{\mathrm{rot} \, 1}\\- \tilde V_{\mathrm{rot} \, 2} & \tilde V_{\mathrm{rot} \, 1} & 0\end{matrix} \right)$$ with $ \tilde  V_{\mathrm{rot}} = V_{\mathrm{rot}} + 2 \nabla^\perp \n.$
\end{theo}
This result can be applied to the interplay of the conserved quantities, and provides an alternate proof of a result by A. Michelat and T. Rivière in \cite{michelatclassifi}.
\begin{cor}
Let $\Phi \, : \, \D \rightarrow \R^3$ be a Willmore immersion, conformal, of conformal Gauss map $Y$.
Let $\iota \, : \, x \mapsto \frac{x}{|x|^2}$ be the inversion at the origin. Let $V_{*, \iota}$ be the conserved quantity corresponding to the transformation $*$ for $\iota \circ \Phi$.
Then 
$$ \begin{aligned}
V_{\mathrm{tra}, \, \iota } &= V_{\mathrm{inv}} \\
V_{\mathrm{inv}, \, \iota } &= V_{\mathrm{tra}} \\
V_{\mathrm{dil}, \, \iota} &= - V_{\mathrm{dil}} \\
\tilde V_{\mathrm{rot}, \, \iota} &= \tilde V_{\mathrm{rot}}.
\end{aligned}$$
\end{cor}

Finally we will study conformally CMC surfaces using a moving frame for the conformal Gauss map. This will yield a notable improvement of theorem \ref{lesresultatsquonavaitdeja}, that we frame in the more elegant framework of immersions in $\s^3$.
\begin{theo}
\label{conformementCMCnya1}
Let $X$ be a smooth conformal immersion from $\D$ to $\s^3$, and $\Phi$ (respectively $Z$) its representation in $\R^3$ (respectively $\mathbb{H}^3$) through $\pi$ (respectively $\tilde \pi$). 
 We assume that $X$ (or equivalently, see (\ref{projectionstereo5}) and (\ref{projectionhyper5}), $\Phi$ or $Z$)  has no umbilic point.
One of the representation of $X$ is conformally $CMC$ in its ambiant space if and only if $\mathcal{Q}$ is holomorphic and $X$ is isothermic.
More precisely  $\left( \mathcal{W}_{\s^3} ( X)\right)^2 - \overline{\omega}^2e^{-4\Lambda} \mathcal{Q}$ is then necessarily real and 
\begin{itemize}
\item
$\Phi$ is conformally CMC (respectively minimal) in $\R^3$ if and only if $$\left( \mathcal{W}_{\s^3} ( X)\right)^2 - \overline{\omega}^2e^{-4\Lambda} \mathcal{Q}=0.$$ 
\item 
$X$  is conformally CMC (respectively minimal) in $\s^3$ if and only if $$\left( \mathcal{W}_{\s^3} ( X)\right)^2 - \overline{\omega}^2e^{-4\Lambda} \mathcal{Q} < 0.$$
\item
$Z$ is conformally CMC (respectively minimal) in $\Hy^3$ if and only if $$\left( \mathcal{W}_{\s^3} ( X)\right)^2 - \overline{\omega}^2e^{-4\Lambda} \mathcal{Q} > 0.$$
\end{itemize}
In particular, conformally minimal immersions satisfy $\mathcal{W}_{\s^3} (X)=0$. 
\end{theo}

In the previous theorem $ \mathcal{W}_{\s^3}(X) =0$ is the Willmore equation of immersions in $\s^3$ (see (\ref{wtorduX}) for a definition and to see it arise organically), $\omega$ is the tracefree curvature of $X$ and $\Lambda$ its conformal factor. While the three conditions were all expressed in terms of immersions in $\s^3$, similar ones could be drawn in terms of immersions in $\R^3$ or $\Hy^3$. We have excluded the case of euclidean spheres (since our immersions are assumed to have no umbilic points) which are both CMC in $\R^3$ and in $\s^3$. This is due to the degenerescence of the conformal Gauss map at umbilic points. 

The added value here is that we can point out in which space the immersion is CMC and that we show $\overline{\omega}^2 \mathcal{Q}$ to be real whenever the immersion is isothermic. More interestingly $\overline{\omega^2} \mathcal{Q} \in \R$ can itself be replaced by the weaker $\overline{\omega^2} \mathcal{Q}$  holomorphic. Indeed  if $\overline{\omega^2} \mathcal{Q} = \varphi$ is holomorphic, $\overline{\omega^2} \frac{\mathcal{Q}}{\varphi} = 1 \in \R$ with $\frac{\mathcal{Q}}{\varphi}$ holomorphic, meaning that  by definition $X$ is isothermic, and thus, according to the theorem, conformally CMC. Moreover since $X$ is conformally CMC, once more according to the theorem, necessarily $\overline{\omega^2} \mathcal{Q}  \in \R$.

This can be somewhat put in perspective  with a discussion in section 5 of \cite{loopiloop}. In this article F. Hélein derives a (non explicit) Weierstrass formula for Willmore surfaces (see also J. Dorfmeister and P. Wang's work in \cite{dorfwangtarxiv1} for another viewpoint and a comparison with Hélein's results), and discusses a particular subcase, giving a characterization of conformally minimal surfaces in terms of his Weierstrass data.  In this characterization one is a real constant  whose sign determines in which space the immersion is conformally minimal, and bears striking resemblance to $ \mathcal{Q} $ in isothermic coordinates. While we cannot, due to the non-explicit nature of the Weierstrass formula, state that $\nu$ is in fact linked to $\mathcal{Q}$, the aforementioned similarities do suggest so. Such an identification would not only shed light on the Weierstrass data, but answer a question raised by F. Hélein as to what $\nu$ non real but constant means.

The novelty in these results lies in their expliciteness, in the characterizations derived on $\mathcal{Q}$ and in the determining of the nature of the space (compared with theorem \label{lesresultatsquonavaitdeja}). %Indeed theorem 4.4 in \cite{bobohle} (by C.Bohle, attributed to K.Voss) already stated a similar fact.

%\begin{theo}
%The Bryant functional of an immersion into $\s^3$ is holomorphic if and only if, locally and away from umbilics and isolated points, the immersion is Willmore or has constant mean curvature with respect to some space form subgeometry.
%\end{theo}

The point of view of the conformal Gauss map is also fruitful when considering the index of Willmore surfaces  (consider for instance \cite{bibpalmerconfgaussmap}).
%which is a very important topic when considering min-max techniques as shown in  \cite{michelatarxiv2} and \cite{michelatarxiv1}. Beside  
This will be the subject of an incoming paper \cite{MM}.

Considering the results obtained in higher codimensions by N. Ejiri (see \cite{MR950596}), S. Montiel (see \cite{MR1695032}) or J. Dorfmeister and P. Wang (see \cite{DPW}) using in no small part integrable systems techniques,  an interesting question arises as to how well our results generalize outside the codimension 1 case.

{\bf Acknowledgments: } The author would like to thank Paul Laurain for his support and advices, and Alexis Michelat for helpful and enlightening conversations.
This work was partially supported by the ANR BLADE-JC.

\section{Conformal Geometry in the three model spaces }
\label{section1}
\subsection{Local conformal equivalences }
\label{localconformalequivalences}
In the following $\langle .,. \rangle$ will denote the standard product on the relevant contextual space. For instance if $u, v \in \R^m$ with $m \in \mathbb{N}$, $\langle u,v \rangle$ denotes the euclidean product of $u$ and $v$ in $\R^m$. If $u$ % = \begin{pmatrix} u_1 \\ u_2 \\ u_3 \\ u_4 \\u_5 \end{pmatrix}$ 
and $v$ %= \begin{pmatrix} v_1 \\ v_2 \\ v_3 \\ v_4 \\v_5 \end{pmatrix}$
 are stated to be in $\R^{m,1}$ then $\displaystyle{ \langle u , v \rangle =\sum_{i=1}^m u_i v_i-u_{m+1}v_{m+1}}$ denotes the $(m,1)$ Lorentzian product of $u$ and $v$ in $\R^{m+1}$.

%In the following $\langle .,. \rangle_m$ will denote the standard product on $\R^m$,  so for instance $\langle.,. \rangle_3$ is the euclidean product on $\R^3$ and $\langle .,. \rangle_{4,1}$  the $(4,1)$ Lorentzian product on $\R^{4,1}$. Should the nature of the product be unambiguous we will omit the index so as not to burden computations. 

We will focus on immersions into the Euclidean space $ \R^3$, into the round sphere $\s^3$ and into the hyperbolic space $\mathbb{H}^3$.

These three spaces are locally conformally equivalent and thus their respective conformal geometry can be linked. Namely the stereographic projection from the north pole $N$

$$\pi : \left\{ \begin{aligned}
  \s^3   \backslash \{ N \} & \rightarrow  \R^3  \\
(x,y,z,t) & \mapsto \frac{1}{1-t} \begin{pmatrix} x \\ y \\ z \end{pmatrix} \end{aligned} \right.$$

is a conformal diffeomorphism whose  inverse is

$$ \pi^{-1} : \left\{ \begin{aligned}
 \R^3  &\rightarrow\s^3   \backslash \{ N \}\\
(x,y,z) & \mapsto \frac{1}{1+ r^2} \begin{pmatrix} 2x \\2 y \\2 z \\ r^2-1 \end{pmatrix}  \end{aligned} \right.$$

which extends to a conformal diffeomorphism $\R^3 \cup \{ \infty \} \rightarrow \s^3$. Consequently one can link $\mathrm{Conf} \left(\R^3\cup \{ \infty \} \right) $ and $\mathrm{Conf} \left( \s^3 \right)$.

\begin{prop}
%[$\mathrm{Conf} \left(\R^3\cup \{ \infty \} \right) \simeq \mathrm{Conf} \left( \s^3 \right)$]
\label{confr3s3}
$\pi$ realises an isomorphism between $\mathrm{Conf} \left( \R^3 \cup  \{ \infty \}\right)$ and $\mathrm{Conf} \left( \s^3 \right)$, with $\mathrm{Conf}(X)$ being the group of conformal diffeomorphisms of $X$.
\end{prop}

The conformal diffeomorphisms of $\R^3$ are well-known and  detailed by the Liouville  theorem (see theorem 1.1.1 of \cite{bibrefconformaldifferentialgeometry}).

\begin{theo}

\label{liouville}
 Any 
%conformal mapping $\varphi$ of $\R^3 \cup \infty$ satisfies 
$\varphi \in \mathrm{Conf}( \R^3 \cup \infty)$ satisfies
$$ \varphi = T_{\vec{a}} \circ R_\Theta \circ D_\lambda $$
if $\varphi \left( \infty \right) = \infty$, 
$$ \varphi = T_{\vec{b}} \circ R_\Theta  \circ D_\lambda \circ \iota \circ T_{\vec{a}}$$
otherwise. Here $T_{\vec{a}}$ and $T_{\vec{b}}$ denote translations, $D_\lambda$  a dilation, $ R_\Theta$ a rotation and $\iota \,: x \mapsto \frac{x}{|x|^2}$ the inversion at the origin.  Such decompositions are unique.
\end{theo}
  
Combining both ensures a description of conformal diffeomorphisms of $\s^3$. 

\begin{cor}
\label{corliouville}
Any conformal mapping $\varphi \in \mathrm{Conf}( \s^3)$ satisfies either

$$ \varphi = \pi^{-1} \circ T_{\vec{b}} \circ R_\Theta \circ D_\lambda  \circ T_{\vec{a}} \circ \pi$$
if $\varphi \left( N\right) = N$,
$$ \varphi = \pi^{-1} \circ T_{\vec{b}} \circ R_\Theta  \circ D_\lambda \circ \iota \circ T_{\vec{a}} \circ \pi$$

otherwise.
\end{cor}

\vspace{5mm}

Using the Poincaré disk model of the hyperbolic space   one finds an isometry \newline $\tilde \pi_0 \, : \,  \mathbb{H}^3 \rightarrow \left( B_1(0), \frac{\langle ., . \rangle}{\left(1- |p|^2\right)^2} \right)$ and thus a conformal diffeomorphism between $\mathbb{H}^3$ and the unit ball of $\R^3$. It will convenient in the following to consider $\mathbb{H}^3$ as the upper part  of the quadric $\left\{ v \in \R^{3,1} \,  \left| \langle v, v \rangle = -1 \right. \right\}$ in $\R^{3,1}$ : $$ \mathbb{H}^3 = \left\{(x,y,z,t) \left| \, x^2 +y^2 +z^2 -t^2 +1 = 0 \text{ and } t \ge 0 \right. \right\} \subset \R^{3,1}.$$ Then the following projection 
%from $( 0,0,0,-1)$
 yields an explicit conformal diffeomorphism
 
$$ \tilde \pi : \left\{ \begin{aligned}
&  \mathbb{H}^3  \rightarrow  B_1(0) \\
& (x,y,z,t)  \mapsto \frac{1}{1+t} \begin{pmatrix} x \\ y \\ z \end{pmatrix} \end{aligned} \right.$$

of inverse

$$ \tilde \pi^{-1} : \left\{ \begin{aligned}
&  B_1(0) \rightarrow \mathbb{H}^3 \\
& (x,y,z)  \mapsto \frac{1}{1-r^2} \begin{pmatrix} 2x \\2 y \\2 z \\ r^2+1 \end{pmatrix}.  \end{aligned} \right.$$

\subsection{The space of spheres of $\s^3$ }
\label{subsectionspaceofpsheres}
In the present subsection we wish to properly represent the geometry of geodesic spheres of $\s^3$.  Our motivation comes from the following result, drawn from chapter 1 in \cite{bibrefconformaldifferentialgeometry}.

\begin{theo}
\label{characterizationconformalmappingsspheres}
Let $(M, g)$ and $(N, h)$ be two Riemann manifolds and $\varphi \, : M \rightarrow N$. 
$\varphi$ is conformal if and only if it sends a geodesic sphere of $M$ into a geodesic sphere of $N$.
\end{theo}
Thanks to theorem \ref{characterizationconformalmappingsspheres}, one would then expect to be able to detail conformal diffeomorphisms of $\s^3$. Moreover since $\mathbb{H}^3 \hookrightarrow \R^3 \hookrightarrow \s^3$ conformally we would subsequently be able to represent geodesic spheres in $\mathbb{H}^3$  and $\R^3$.

%\vspace{5mm}
The stereographic projection ensures that $\R^{3} \cup\{ \infty \} \simeq \s^3$ conformally, and thus geodesic spheres in $\s^3$ are images by $\pi^{-1}$ of euclidean spheres and planes ("spheres" going through $\infty$) of $\R^3$. They will be called spheres in $\s^3$. More precisely :

\begin{de}
%[Spheres in $\s^3$]
A sphere in $ \s^3$ is equivalently defined as follows : 
\begin{itemize}
\item 
The inverse of the stereographic projection of  a sphere or a plane in $\R^3$.
\item 
$\{ x \in \s^3 \quad d (q, x ) = r \}$ for a given $q$ in  $\s^3$. $q$ is then the center of the sphere, of radius $r \le \frac{\pi}{2}$.
\end{itemize}
An equator of $\s^3$ is a sphere of maximum radius  $r = \frac{\pi}{2}$.
\end{de}

%Using spherical coordinates one can parametrize  a given sphere $\sigma$ of $\s^3$ : $$\sigma = \left\{  \begin{pmatrix} \cos \left( \varphi_0 + \varphi \right)  \\ \sin\left( \varphi_0 + \varphi \right) \cos \left( \psi_0 + \psi \right) \\ \sin \left( \varphi_0 + \varphi \right) \sin \left( \psi_0 + \psi \right) \cos \left( \theta_0 + \theta \right) \\ \sin \left( \varphi_0 + \varphi \right) \sin \left( \psi_0 + \psi \right) \cos \left( \theta_0 + \theta \right) \end{pmatrix} \text{ s.t. } q =  \begin{pmatrix} \cos \left( \varphi_0  \right)  \\ \sin\left( \varphi_0 \right) \cos \left( \psi_0 \right) \\ \sin \left( \varphi_0  \right) \sin \left( \psi_0  \right) \cos \left( \theta_0 \right) \\ \sin \left( \varphi_0  \right) \sin \left( \psi_0  \right) \cos \left( \theta_0  \right) \end{pmatrix} \right\}$$

One can easily check that spheres in $\s^3$ are orientable.

\begin{de}
%[Space of spheres]
Let $$ \begin{aligned} \mathbb{M}_0 &=  \{ \mathrm{non}\text{-}\mathrm{oriented}\text{ }\mathrm{spheres}\text{ }\mathrm{in}\text{ }  \s^3 \}, \\ \mathbb{E}_0 &= \{ \mathrm{non}\text{-}\mathrm{oriented}\text{ }\mathrm{equatorial}\text{ }\mathrm{spheres}\text{ }\mathrm{in}\text{ } \s^3 \}, \end{aligned} $$ and $$ \begin{aligned} \mathbb{M} &=  \{ \mathrm{oriented}\text{ }\mathrm{spheres}\text{ }\mathrm{in}\text{ } \s^3 \}, \\ \mathbb{E} &= \{ \mathrm{oriented}\text{ }\mathrm{spheres}\text{ }\mathrm{in}\text{ } \s^3 \}. \end{aligned} $$
\end{de}

Let $ \sigma $ be a non-oriented sphere of radius $r < \frac{\pi}{2}$.
Let $X_{\sigma} \in \sigma$ be any point on the sphere and $\vec{N}_\sigma$ the inward  pointing (relative to $\sigma$) normal to $\sigma$ at $X_{\sigma}$. Then  $p_\sigma = X_\sigma + \tan r \vec{N}_\sigma$ is the summit of the tangent cone to $\s^3$ along $\sigma$. Since a sphere in $\s^3 $ has constant mean curvature $h = \frac{1}{\tan r}$ (see (\ref{h(r)}) in the appendix \ref{meancurvsphere} ), $p_\sigma = X_\sigma + \frac{1}{h}\vec{N}_\sigma$. This gives us a representation of $\mathbb{M}_0 \backslash \mathbb{E}_0 $ : 

$$ P_0 : \left\{ \begin{aligned}
&\mathbb{M}_0 \backslash \mathbb{E}_0 \rightarrow  \R^4 \backslash B_1(0) \\
&\sigma \mapsto p_\sigma, \end{aligned} \right.$$
as shown in figure 1.
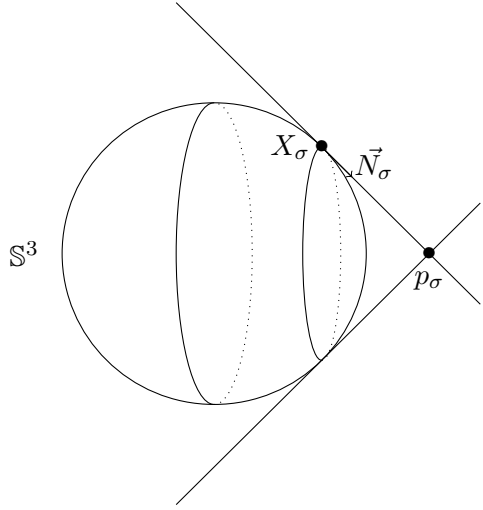
\begin{figure}[!h]
\centering
\begin{tikzpicture}
\draw(0.,0.) circle (2.cm);
%\draw (0, 2).. controls (-0.5,0.5) and (-0.5,-0.5) .. (0,-2);
%\draw[dotted] (0, 2).. controls (0.5,0.5) and (0.5,-0.5) .. (0,-2);
%\draw [domain=-2:2] plot(0,\x);
%\draw [domain=-1.4142135623731:1.4142135623731] plot(1.4142135623731, \x );
\draw (1.4142135623731, 1.4142135623731) arc (90:270: 0.25 and 1.4142135623731);
\draw[dotted] (1.4142135623731, -1.4142135623731) arc (270:450: 0.25 and 1.4142135623731);
\draw (1.4142135623731, 1.4142135623731) node {$\bullet$};
\draw (1, 1.4142135623731) node {$X_\sigma$};
\draw[domain= -0.5:3.5] plot(\x,-\x +2.8284271247462);
\draw [->] (1.4142135623731, 1.4142135623731)-- (1.8142135623731,1.0142135623731);
\draw (2.1,1.2) node {$\vec{N_\sigma}$};
\draw[domain= -0.5:3.5] plot(\x,\x -2.8284271247462);
\draw (2.8284271247462,0) node {$\bullet$};
\draw (2.8284271247462,-0.1) node[below] {$p_\sigma$};
\draw (-2.5 , 0) node {$\mathbb{S}^3$};
%\draw [shift={(6.404877049180328,0)}] plot[domain=4.409720912080944:5.015057048688435,variable=\t]({0.*6.709877049180328*cos(\t r)+1.*6.709877049180328*sin(\t r)},{1.*6.709877049180328*cos(\t r)+0.*6.709877049180328*sin(\t r)});
\draw (0,2) arc (90:270: 0.5cm and 2cm);
\draw[dotted] (0,-2) arc (270:450: 0.5cm and 2cm);
%\draw [domain=-2.79:2.96] plot(\x,{(-0.27979366493895363--0.06522205548878501*\x)/-0.2356345566706397});
%\draw [domain=-2.79:2.96] plot(\x,{(-0.2995677920404056--0.37621048614319563*\x)/-0.5474331756635662});
%\draw (-1.525,1.69) node[above] {$F(\mathbb{S}_{x_0}(r) )$};
%\draw (-1.523,1.62) node {$\bullet$};
%\draw (-0.765,1.05) node[below] {$\phi( \mathbb{S}_{x_0}(r))$};
\end{tikzpicture}
\centering
\caption{Construction of $p_\sigma$.}
\end{figure}

Conversely given any $p \in \R^4 \backslash B_1(0)$ there exists a unit cone of summit $p$ tangent to $\s^3$, along a sphere of $\s^3$. $P_0$ is then a bijection.

As $\sigma$ becomes equatorial, $ h \rightarrow 0$, meaning $p \rightarrow \infty$ and $\vec{N}_\sigma\rightarrow \vec{\nu}$ with $\vec{\nu} \in \R^4$ independant on the chosen $X_\sigma$. To properly represent all of $\mathbb{M}_0$ we define $\underline{p_\sigma} = \begin{pmatrix} p_\sigma \\1 \end{pmatrix}$. Then 

$$ \frac{ \underline{p_\sigma} }{ | {p_\sigma} | } = \begin{pmatrix} \frac{p_\sigma}{|p_\sigma|}\\ \frac{1}{|p_\sigma|} \end{pmatrix}\rightarrow \begin{pmatrix} \vec{\nu} \\ 0 \end{pmatrix}$$ as $\sigma$ tends toward an equatorial sphere of constant normal $\vec{\nu}$.
%, with $\nu$ the normal to the limit sphere in $\s^3$.

Then one can represent $\mathbb{M}_0$ in $ \mathbb{RP}^4$, with equatorial spheres being sent to $  \left[ \begin{pmatrix} \vec{\nu}  \\0 \end{pmatrix} \right]$ typed directions (where $[d]$ denotes the direction of $d \in \R^5$).
% the $\s^3$-shaped infinity corresponding to equatorial spheres : 

$$ P_1 : \left\{ \begin{aligned}
&\mathbb{M}_0 \rightarrow  \mathbb{RP}^4  \\
&\sigma \mapsto \left[ \underline{p_\sigma} \right]. \end{aligned} \right.$$

Since any equatorial sphere is fully determined by its normal, $P_1$ remains injective. However for non equatorial spheres  $p_\sigma$ is necessarily outside $B_1(0)$, and thus $P_1$ cannot be surjective. 

Pursuing will require some basic notions in semi-Riemannian geometry.

\begin{de}
%[Lorentzian notions]
Let $m \in \mathbb{N}$ and $v \in \R^{m,1}$. Then $v$ is said to be
\begin{itemize}
\item \emph{spacelike} if  $\langle v,v \rangle >0$,
\item \emph{lightlike} if $\langle v,v \rangle = 0$,
\item \emph{timelike} if $\langle v,v \rangle <0$.
\end{itemize}

Accordingly a direction $d \in \mathbb{RP}^{m+1}$ is called 
\begin{itemize}
\item \emph{spacelike} if  there exists $v \in \R^{m,1}$ such that $\langle v,v \rangle >0$ and $[v]=d$,
\item \emph{lightlike}  if  there exists $v \in \R^{m,1}$ such that $\langle v,v \rangle =0$ and $[v]=d$,
\item \emph{timelike}  if  there exists $v \in \R^{m,1}$ such that $\langle v,v \rangle <0$ and $[v]=d$.
\end{itemize}

We also define 
\begin{itemize}
\item the \emph{De Sitter} space of $\R^{m,1}$ as the set of unit spacelike vectors. \newline It will be denoted $\s^{m,1} :=  \left\{ v \in \R^{m,1} \quad \langle v,v \rangle =1 \right\}$,
\item the \emph{isotropic cone} of $\R^{m,1}$ as the set of lightlike vectors. \newline It will be denoted $ \mathcal{C}^{m,1} := \left\{ v \in \R^{m,1} \quad \langle v,v\rangle =0 \right\}$.
\end{itemize}
\end{de}

This definition is illustrated by the following figure.

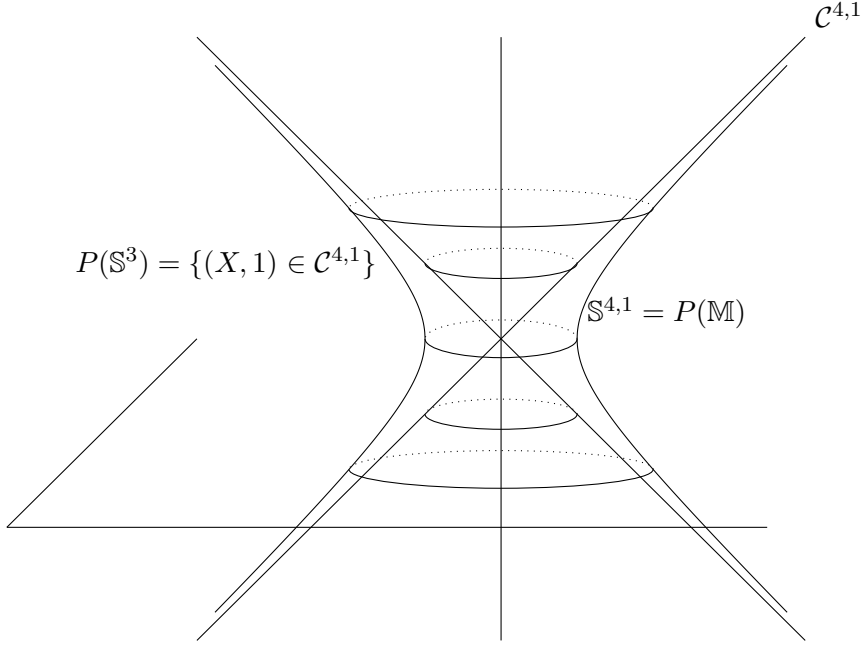
\begin{figure}[!h]
\centering
\begin{tikzpicture}
\draw[domain=0:2,smooth,variable=\t]
plot ({(exp(\t)+exp(-\t))/2},{(exp(\t) -exp(- \t))/2});
\draw[domain=-2:0,smooth,variable=\t]
plot ({(exp(\t)+exp(-\t))/2},{(exp(\t) -exp(- \t))/2});
\draw[domain=0:2,smooth,variable=\t]
plot ({-(exp(\t)+exp(-\t))/2},{(exp(\t) -exp(- \t))/2});
\draw[domain=-2:0,smooth,variable=\t]
plot ({-(exp(\t)+exp(-\t))/2},{(exp(\t) -exp(- \t))/2});
\draw (-1,0) arc (180:360:1cm and 0.25cm);
\draw[dotted] (1,0) arc (0:180:1cm and 0.25cm);
\draw (-2,1.7320508075689) arc (180:360:2cm and 0.25cm);
\draw[dotted] (2,1.7320508075689) arc (0:180:2cm and 0.25cm);
\draw (-2,-1.7320508075689) arc (180:360:2cm and 0.25cm);
\draw[dotted] (2,-1.7320508075689) arc (0:180:2cm and 0.25cm);

\draw (-4,-4) -- (0,0) ; \draw (0,0)--(4,4);
\draw (-4,4) -- (0,0); \draw (0,0)--(4,-4);
%\draw (-4,0)--(4,0);
\draw (0,-4)-- (0,4);
\draw (-1,1) arc (180:360:1cm and 0.2cm);
\draw[dotted] (1,1) arc (0:180:1cm and 0.2cm);
\draw (-1,-1) arc (180:360:1cm and 0.2cm);
\draw[dotted] (1,-1) arc (0:180:1cm and 0.2cm);

\draw (1,0) node[above right] {$\mathbb{S}^{4,1} = P( \mathbb{M})$};
\draw (4,4) node[above right] {$\mathcal{C}^{4,1}$};
%\draw (4,0) node[above] {$\mathbb{R}^4$};
%\draw (0,-4) node[left] { $\mathbb{R}$};
%\draw[red] (-1,0) node {$\bullet$};
%\draw[red] (-1,0) node[below left] {$P(\mathbb{S}) = (0, -1,0)$};
%\draw[blue] (-1,1) node[left] {$P( \{  0 \}) = (0,-1,1)$};
%\draw[blue] (-1,1) node {$\bullet$};
\draw (-1.5,1) node[left] {$P( \mathbb{S}^3) =  \{ (X,1) \in \mathcal{C}^{4,1} \}$};
\draw (-4,0)--(-6.5,-2.5);
\draw (-6.5,-2.5)--(3.5,-2.5);
\centering
\end{tikzpicture}
\caption{De Sitter and the isotropic cone.}
\end{figure}

One can realize the  image of $P_1$ is the set of all the space-like directions of $\R^{4,1}$  which is isomorphic to $\s^{4,1} / \{ \pm Id \}$.

We finally obtain our representation of non-oriented spheres : 

$$ P : \left\{ \begin{aligned}
\mathbb{M}_0  &\rightarrow \s^{4,1} / \{ \pm Id \} \\
\sigma &\mapsto  \frac{\underline{p} }{ \left\| \underline{p} \right\|} %= h \begin{pmatrix} X_\sigma \\ 1 \end{pmatrix} + \begin{pmatrix}\vec{N}_\sigma \\ 0 \end{pmatrix} 
\end{aligned} \right.$$

%for any $X_\sigma \in \sigma$, 
where $\left\|\underline{p} \right\| = \sqrt{ \langle \underline{p_\sigma}, \underline{p_\sigma} \rangle }$.

$P$  is easily extended to $\mathbb{M}$ by taking the natural two covering of $ \s^{4,1} / \{ \pm Id \}$. Two opposite points in the De Sitter space then represent the same sphere with opposite orientations.

\begin{equation}
\label{s3}
 P : \left\{ \begin{aligned}
&\mathbb{M}  \rightarrow \s^{4,1}  \\
&\sigma \mapsto   \frac{\underline{p} }{ \left\| \underline{p} \right\|}= h \begin{pmatrix} X_\sigma \\ 1 \end{pmatrix} + \begin{pmatrix} \vec{N}_\sigma \\ 0 \end{pmatrix} \end{aligned} \right.
\end{equation}

for any $X_\sigma \in \sigma$. 

As $ h \rightarrow \infty$ (that is the radius of the sphere goes to $0$ and thus the sphere collapses on a point $X \in \s^3$), $\frac{P(\sigma)}{h} \rightarrow (X, 1)$, meaning that $P(\sigma)$ tends to $\infty$ in an isotropic direction of $\R^{4,1}$ bijectively and smoothly linked with the point of collapse $X$. One can then continuously extend $P$

\begin{equation}
\label{s3}
 P : \left\{ \begin{aligned}
&\mathbb{M} \cup \s^3  \rightarrow \s^{4,1} \cup \mathcal{C}^{4,1}  \\
&\sigma \in \mathbb{M} \mapsto     \frac{\underline{p} }{ \left\| \underline{p} \right\|} = h \begin{pmatrix} X_\sigma \\ 1 \end{pmatrix} + \begin{pmatrix}\vec{N}_\sigma \\ 0 \end{pmatrix} \in \s^{4,1} \\
& X \in \s^3 \mapsto \begin{pmatrix} X \\ 1 \end{pmatrix}  \in \mathcal{C}^{4,1}.  \end{aligned} \right.
\end{equation}

\subsection{The space of generalized spheres of $\R^3$}
Since the stereographic projection is  a conformal diffeomorphism, the set of non-oriented (respectively oriented) spheres and planes or $\R^3$ is in bijection with $\mathbb{M}_0$ (respectively $\mathbb{M}$) and can be represented using $P$.  Using  formula (\ref{projectionstereo7}) (see appendix \ref{sectionannexcalculss3}) one finds  

\begin{equation}
\label{r3}
 P : \left\{ \begin{aligned}
&\left( \R^3\cup \{\infty\} \right) \cup\mathbb{M}   \rightarrow \s^{4,1} \cup \mathcal{C}^{4,1}  \\
&\sigma \in \mathbb{M} \mapsto H_\sigma \begin{pmatrix} \Phi_\sigma \\ \frac{|\Phi_\sigma |^2 -1}{2} \\ \frac{|\Phi_\sigma|^2 +1}{2} \end{pmatrix} + \begin{pmatrix} \n_\sigma \\ \langle \n_\sigma, \Phi_\sigma \rangle \\ \langle \n_\sigma, \Phi_\sigma \rangle \end{pmatrix} \text{ for any } \Phi_\sigma \in \sigma \\
& \Phi \in \R^{3} \mapsto \begin{pmatrix} \Phi \\ \frac{ |\Phi|^2 - 1 }{2} \\ \frac{|\Phi|^2 +1}{2} \end{pmatrix} \in \mathcal{C}^{4,1} \\
& \infty \mapsto \begin{pmatrix} 0, 1,1 \end{pmatrix} \in \mathcal{C}^{4,1}.\end{aligned} \right.
\end{equation}

\subsection{The space of generalized spheres of $\Hy^3$}

Similarly consider $\mathbb{M}_{\Hy^3}$ the set of oriented geodesic spheres in $\Hy^3$. The function $\pi^{-1} \circ \tilde \pi$ sends $\Hy^3$ injectively into $\s^3$  and thus maps $\mathbb{M}_{\Hy^3}$ injectively into $\mathbb{M}$. $\mathbb{M}_{\Hy^3}$ can then be represented using $P$ (see formula (\ref{projectionhyper7}) in appendix \ref{sectionannexcalculsH3}) one finds  
\begin{equation}
\label{h3}
 P : \left\{ \begin{aligned}
&\Hy^3 \cup \mathbb{M}_{\Hy^3}  \rightarrow \s^{4,1}  \cup \mathcal{C}^{4,1}\\
&\sigma  \in \mathbb{M}_{\Hy^3}\mapsto H_\sigma^Z \begin{pmatrix} Z_{h \, \sigma} \\-1 \\  Z_{ 4 \, \sigma }\end{pmatrix} + \begin{pmatrix}  \n^Z_{ h \, \sigma} \\ 0 \\ \n^Z_{ 4 \, \sigma } \end{pmatrix} \text{ for any } \begin{pmatrix}  Z_{h\, \sigma}  \\  Z_{ 4 \, \sigma} \end{pmatrix} \in \sigma \\
& Z= \begin{pmatrix} Z_h \\ Z_4 \end{pmatrix} \in \Hy^3 \mapsto  \begin{pmatrix} Z_{h}\\ -1 \\ Z_4 \end{pmatrix} \in \mathcal{C}^{4,1}. \end{aligned} \right.
\end{equation}

\subsection{  $\mathrm{Conf}( \s^3) \simeq SO(4,1)$ }
As foreshadowed in subsection \ref{subsectionspaceofpsheres}, we can use $P$ to study conformal diffeomorphisms of $\s^3$.

\begin{theo}
$P$ realises an isomorphism between $\mathrm{Conf}( \s^3)$ and $SO(4,1)$.
\end{theo}
\begin{proof}

According to proposition \ref{confr3s3}, showing $ \mathrm{Conf} \left( \R^3 \cup \{\infty \} \right) \simeq SO(4,1)$ is enough.  We will proceed in three steps : we will start by defining the correspondance, show that it represents  a morphism and conclude by proving it is bijective.

\begin{itemize}

\item

\underline{ \bf{ Step 1 : }} {\bf Defining the correspondance $ M \rightarrow \varphi_M$  }

The core idea here is that isotropic directions in $\R^{4,1}$ are in bijection with $\R^3 \cup \{ \infty \}$, and that any $M \in SO(4,1)$ shuffles them. Thus $M$ yields a transformation of $\R^3\cup \{\infty \}$. Its conformality is all one needs to prove.

Let $p(x) := \begin{pmatrix} x \\ \frac{|x|^2-1}{2} \\ \frac{|x|^2+1}{2} \end{pmatrix} = P_{|_{\R^3 \cup \{ \infty \} }}  (x)$. One easily shows that for all $i,j$  : $$\langle \partial_i p , \partial_j p \rangle = \delta_{ij},$$

that is $p : \, \R^3 \rightarrow P\left( \R^3 \cup \{ \infty \} \right)$ is an isometry. As $x \rightarrow \infty$, $\frac{p(x)}{|p(x)|} \rightarrow \begin{pmatrix} 0 \\ 1 \\1 \end{pmatrix}$. Noticing that $P\left( \R^3 \cup \{ \infty \} \right) = \left\{ p \in \mathcal{C}^{4,1} \text{ s.t. } p_5- p_4 = 1 \right\} \cup \{ (0,0,0,1,1 ) \}$, one can conversely associate to any $p \in \mathcal{C}^{4,1}$ a point $x = \frac{\left( p_1,p_2,p_3 \right) }{ p_5 - p_4} \in \R^3\cup \{\infty \}$ depending only on the direction of $p$.

Given $M \in SO(4,1)$ let $  y= Mp(x) = \begin{pmatrix} y_{\diamond} \\ y_4 \\y_5 \end{pmatrix} $. Then

$$ \begin{aligned}
\langle y_{\diamond}, y_{\diamond} \rangle &= y_5^2-y_4^2  \\
\langle \partial_i y_{\diamond}, y_{\diamond} \rangle &= \partial_i y_5 y_5 - \partial_i y_4 y_4 \\
\langle \partial_i y_{\diamond}, \partial_j y_{\diamond} \rangle &= \delta_{ij} +\partial_i y_5 \partial_j y_5 - \partial_i y_4 \partial_j y_4.
\end{aligned}$$

Renormalizing as suggested, let  $\varphi_M (x) = \frac{y_{\diamond}}{y_5-y_4} = p^{-1} \left( \frac{M p(x) }{ \left( M p(x) \right)_5- \left( M p(x) \right)_4} \right) $. $\varphi_M$ is a transformation of $\R^3 \cup \{ \infty \}$. Let us show it is conformal :

$$ \begin{aligned}
\langle \partial_i \varphi_M , \partial_j \varphi_M \rangle &=\left\langle  \frac{\partial_i y_{\diamond}}{y_5-y_4}- \frac{\left( \partial_i y_5 - \partial_i y_4 \right)  y_{\diamond}}{\left( y_5-y_4 \right)^2},  \frac{\partial_j y_{\diamond}}{y_5-y_4}- \frac{\left( \partial_j y_5 - \partial_j y_4 \right)  y_{\diamond}}{\left( y_5-y_4 \right)^2} \right\rangle \\
&=\frac{1}{ \left( y_5-y_4 \right)^2} \left\langle \partial_i y_{\diamond}, \partial_j y_{\diamond} \right\rangle  + \frac{\left( \partial_i y_5 - \partial_i y_4 \right) \left( \partial_j y_5 - \partial_j y_4 \right)}{ \left( y_5-y_4 \right)^3}    \langle y_{\diamond}, y_{\diamond} \rangle  \\ &
 - \frac{1}{ \left( y_5-y_4 \right)^3}  (\left( \partial_i y_5 - \partial_i y_4 \right) \langle \partial_j y_{\diamond} , y_{\_} \rangle - \left( \partial_j y_5 - \partial_j y_4 \right) \langle \partial_i y_{\diamond} , y_{\diamond} \rangle ) \\
&= \frac{\delta_{ij}}{ \left( y_5-y_4 \right)^2} + \frac{\partial_i y_5 \partial_j y_5 - \partial_i y_4 \partial_j y_4}{ \left( y_5-y_4 \right)^2} \\&+ \frac{1}{ \left( y_5-y_4 \right)^3}   \left( \partial_i y_5 - \partial_i y_4 \right) \left( \partial_j y_5 - \partial_j y_4 \right) \left( \partial_i y_5 y_5 - \partial_i y_4 y_4  \right) \\ 
 &- \frac{ \left( \partial_i y_5 - \partial_i y_4 \right) \left( \partial_j y_5 y_5 - \partial_j y_4 y_4\right)}{ \left( y_5-y_4 \right)^3}   \\
&  + \frac{ \left( \partial_j y_5 - \partial_j y_4 \right) \left(  \partial_i y_5 y_5 - \partial_i y_4 y_4 \right) }{ \left( y_5-y_4 \right)^3}\\
&=\frac{\delta_{ij}}{ \left( y_5-y_4 \right)^2}.
\end{aligned}$$
 Then $\varphi_M  \in \mathrm{Conf}(\R^3)$.
% is a conformal application of $\R^3 \cup \{\infty \}$. Using Liouville's theorem (see theorem \ref{liouville} above) $\varphi \in \mathrm{Conf}(\R^3)$, which concludes this step.

\item
\underline{ \bf{ Step 2 : }} {\bf $M \rightarrow \varphi_M$  is a morphism }

Given $M_1$ and $M_2 \in SO(4,1)$ we compute

$$\begin{aligned}
 \varphi_{M_1} \circ \varphi_{M_2} (x) &=p^{-1} \left( \frac{ M_1 \frac{ M_2 p(x) }{ \left( M_2 p(x) \right)_5 - \left( M_2 p(x) \right)_4} }{ \left( M_1  \frac{ M_2 p(x) }{ \left( M_2 p(x) \right)_5 - \left( M_2 p(x) \right)_4} \right)_5 - \left( M_1  \frac{ M_2 p(x) }{ \left( M_2 p(x) \right)_5 - \left( M_2 p(x) \right)_4} \right)_4 } \right) \\
&= p^{-1} \circ \left( \frac{M_1 M_2 p(x) }{ \left( M_1 M_2 p(x) \right)_5- \left( M_1M_2 p(x) \right)_4} \right) \\
&=  \varphi_{M_1 M_2}(x).
\end{aligned}$$

Thus $M \mapsto \varphi_M$ is a morphism between $ SO(4,1)$ and $\mathrm{Conf} \left( \R^3 \cup \{ \infty \} \right)$.
\item
\underline{ \bf{ Step 3 : }} {\bf  $ M \rightarrow \varphi_M$  is an isomorphism}

Bijectivity is the only property left to show. According to theorem \ref{liouville}, exhibiting $M \in SO(4,1)$ for dilations, translations, rotations and the inversion is enough to ensure surjectivity. Computing   we find

\noindent { \bf Dilations :}

For $D_\lambda (x) =  e^\lambda x$,
 \begin{equation}
\label{reprdilations}
M_{D_{\lambda}} = \begin{pmatrix} \mathrm{Id} & 0 & 0 \\ 0 & \mathrm{ch} \lambda & \mathrm{sh} \lambda \\ 0 & \mathrm{sh} \lambda & \mathrm{ch} \lambda \end{pmatrix} \in SO(4,1).
\end{equation}

\noindent { \bf Rotations :}

For $ R_\Theta (x) = \Theta x $, with $\Theta \in O(3)$,  
\begin{equation}
\label{reprrotations}
M_{R_\Theta} = \begin{pmatrix} \Theta & 0 & 0 \\ 0 & 1 & 0 \\ 0 & 0 & 1 \end{pmatrix} \in SO(4,1) .
\end{equation}

\noindent { \bf Inversion : }

For $\iota (x) = \frac{x}{|x|^2}$, 
\begin{equation}
\label{reprinversion}
M_{\iota} = \begin{pmatrix} -Id & 0 & 0 \\ 0 & 1 & 0 \\ 0 & 0 & -1 \end{pmatrix} \in SO(4,1).
\end{equation}

\noindent { \bf Translations : }
For $T_{\vec{a}}(x) = x + \vec{a}$, with $\vec{a} \in \R^3$, 
\begin{equation}
\label{reprtranslations}
M_{T_{\vec{a}}} = \begin{pmatrix} Id &-  \vec{a} & \vec{a} \\ \vec{a}^T &1- \frac{|\vec{a}|^2}{2} & \frac{|\vec{a}|^2}{2} \\ \vec{a}^T & - \frac{|\vec{a}|^2}{2} &  1+ \frac{|\vec{a}|^2}{2} \end{pmatrix} \in SO(4,1).
\end{equation}

$M \rightarrow \varphi_M$ is then surjective. With injectivity stemming from the uniqueness of the decomposition in theorem \ref{liouville}, $M \rightarrow \varphi_M$ is bijective, which concludes the proof.
\end{itemize}
\end{proof}

A direct consequence of the proof is the explicit formula for the conformal actions of  $ SO(4,1)$ on $\s^3$ and $\R^3$.

\begin{cor}[Action of $SO(4,1)$ on $\R^3$ and $\s^3$]
\label{actionr3s3}
$SO(4,1)$ acts transitively through conformal diffeomorphisms on

\begin{itemize}
\item
 $\s^3$  : 
$$  M.X =  \frac{V_\circ}{V_5}  $$ where $$V = M \begin{pmatrix} X \\ 1 \end{pmatrix} = \begin{pmatrix} V_\circ \\ V_5 \end{pmatrix}.$$
\item
 $\R^3$  : 
$$    M.x = \frac{y_\diamond}{y_5 -y_4} $$ where $$y =  M \begin{pmatrix} x \\ \frac{|x|^2-1}{2} \\ \frac{|x|^2+1}{2} \end{pmatrix}  = \begin{pmatrix} y_\diamond \\ y_4 \\y_5 \end{pmatrix}.$$ 
%\item 
%$SO(4,1)$ acts transitively through conformal diffeomorphisms on $\Hy^3$  : 
%$$  \text{ for } M \in SO(4,1) \quad M.Z = \frac{\left( M \begin{pmatrix} Z_{123} \\ 1 \\Z_4\end{pmatrix} \right)_{1234} }{ \left( M  \begin{pmatrix} Z_{123} \\ 1 \\Z_4\end{pmatrix} \right)_4}.$$
\end{itemize}
\end{cor}

While $\mathrm{Conf}(\s^3) \simeq SO(4,1)$ is well known, the explicit action of $SO(4,1)$ on elements of $\s^3$ is less commonly found.

\subsection{Representations in the three conformal models}
\label{representationsinthethreemodels}
Let $\Sigma$ be a Riemann surface. Let $\Phi \, : \, \Sigma \rightarrow \R^3$ an immersion. Let $g$ be the induced metric, $\n$ be its Gauss map, $H$ its mean curvature and $\Ar$ its tracefree second fundamental form defined as 

\begin{equation} \label{tracefreesecondfundamentalform} \Ar =g^{-1} A - \frac{H}{2} Id \end{equation}
with $A = \left\langle \nabla_g^2 \Phi, \n \right\rangle$ the second fundamental form.

We refer to $X = \pi^{-1} \circ \Phi $ as the representation of $\Phi $ in $\s^3$ and $Z = \tilde \pi^{-1} \circ \Phi$ as the representation of $\Phi$ in $\Hy^3$ (whenever $\Phi ( \Sigma) \subset B_1(0)$). We will often decompose $Z= (Z_h, Z_4)$ with $Z_h=(Z_1,Z_2,Z_3)$.

\section{The Conformal Gauss map}
\label{section2}
The previous considerations on the representation of spheres in the de Sitter space can be applied to the study of the geometry of immersed surface through the conformal Gauss map. To lighten notations, we will denote 

$$\begin{aligned}&p(\Phi) =  \begin{pmatrix} \Phi \\ \frac{ |\Phi|^2 -1}{2} \\ \frac{ |\Phi|^2+1}{2} \end{pmatrix} \text{ for } \Phi \in \R^3, \\
&p(X) = \begin{pmatrix} X \\1 \end{pmatrix} \text{ for } X \in \s^3,\\
&p(Z) = \begin{pmatrix} Z_h \\ -1 \\ Z_4 \end{pmatrix} \text{ for } Z = \begin{pmatrix} Z_h \\ Z_4 \end{pmatrix} \in \Hy^3. \end{aligned}$$

\subsection{Enveloping spherical congruences }
We first introduce  the notion of enveloping spherical congruences.
\begin{de}
Let $\Sigma$ be a Riemann surface. A spherical congruence on $\Sigma$ is a smooth application $ Y  \, : \, \Sigma \rightarrow \s^{4,1}$, that is, a family of oriented spheres parametrized on $\Sigma$. 
Given $\Phi \, : \, \Sigma \rightarrow \R^3$, or equivalently  $X$ its representation in $\s^3$,  or $Z=(Z_h,Z_4)$ in $\Hy^3$, $Y$ envelopes $\Phi$, or equivalently $X$ or $Z $, if and only if 
\begin{equation}
\label{enveloppe1phi}
 \left\langle Y, p(\Phi) \right\rangle = 0 
\end{equation}
and
\begin{equation}
\label{enveloppe2phi}
\left\langle Y, \nabla p(\Phi)  \right\rangle = 0,
\end{equation}
or equivalently
\begin{equation}
\label{enveloppe1X}
\left\langle Y,p(X) \right\rangle = 0 
\end{equation}
and
\begin{equation}
\label{enveloppe2X}
\left\langle Y, \nabla p(X) \right\rangle = 0,
\end{equation}
or 
\begin{equation}
\label{enveloppe1Z}
\left\langle Y,p(Z) \right\rangle = 0  
\end{equation}
and 
\begin{equation}
\label{enveloppe2Z}
\left\langle Y, \nabla p(Z) \right\rangle = 0.
\end{equation}
Geometrically speaking $Y$ envelopes $\Phi$ at the point $p \in \Sigma$ if the generalized sphere $Y(p)$ is tangent to $\Phi( \Sigma)$ at the point $\Phi(p)$.
\end{de}

\begin{proof}
Since $p(\Phi)$, $p(X)$ and $p(Z)$ are pairwise colinear, one finds (\ref{enveloppe1phi}), (\ref{enveloppe1X}) and (\ref{enveloppe1Z}) to be equivalent.

Moreover, assuming (\ref{enveloppe1phi}), (\ref{enveloppe1X}), and  (\ref{enveloppe1Z}),  one deduces

$$\left\langle Y, \nabla p(\Phi) \right\rangle =  \nabla \left( \left\langle Y, p(\Phi) \right\rangle \right) - \left\langle \nabla Y,p(\Phi) \right\rangle = - \left\langle \nabla Y,p(\Phi) \right\rangle,$$

$$\left\langle Y, \nabla p(X) \right\rangle =  \nabla \left( \left\langle Y, p(X) \right\rangle \right) - \left\langle \nabla Y,p(X) \right\rangle = - \left\langle \nabla Y, p(X) \right\rangle$$

and

$$\left\langle Y, \nabla p(Z) \right\rangle =  \nabla \left( \left\langle Y,  p(Z) \right\rangle \right) - \left\langle \nabla Y, p(Z) \right\rangle = - \left\langle \nabla Y,p(Z) \right\rangle,$$

which ensures that (\ref{enveloppe2phi}), (\ref{enveloppe2X}) and (\ref{enveloppe2Z}) are equivalent.
\end{proof}

\begin{ex}[The conformal Gauss map]
\label{ouonintroduitlaGaussconforme}
Let $\Sigma$ be a Riemann surface and $\Phi \, : \, \Sigma \rightarrow \R^3$. The conformal Gauss map $Y$, which to a point $z \in \Sigma$ associates the tangent sphere to the surface at $\Phi(z)$ of center $\Phi(z) + \frac{\n(z)}{|H(z)|}$
% $ \frac{1}{|H(p)|}$ such that the normal to $\Phi(\Sigma)$ at p $\n(p)$ points inside the sphere
 if  $H(z) \neq 0$, and the tangent plane if $H(z)=0$, is a spherical congruence enveloping $\Phi$.

$Y$ can be written as : 

\begin{equation}
\label{Yphi}
 \begin{aligned}
Y &= H \begin{pmatrix} \Phi \\ \frac{ |\Phi|^2 -1}{2} \\ \frac{|\Phi|^2 +1}{2} \end{pmatrix} + \begin{pmatrix} \n \\ \langle \n , \Phi \rangle \\ \langle \n, \Phi \rangle \end{pmatrix}.
\end{aligned}
\end{equation}

One can notice : 
 $$ \begin{aligned}
\left\langle  p(\Phi) , p(\Phi)  \right\rangle &= 0 \\
\left\langle  p(\Phi)  , \nabla p(\Phi)   \right\rangle &= 0 \\
\left\langle  p(\Phi)  , \begin{pmatrix} \n \\ \langle \n , \Phi \rangle \\ \langle \n, \Phi \rangle \end{pmatrix}  \right\rangle &= 0 
\end{aligned}$$
and  in local coordinates
\begin{equation}
\label{lenabladuY} \begin{aligned}
\nabla Y &= \nabla H \begin{pmatrix} \Phi \\ \frac{ |\Phi|^2 -1}{2} \\ \frac{|\Phi|^2 +1}{2} \end{pmatrix}  + H \nabla \begin{pmatrix} \Phi \\ \frac{ |\Phi|^2 -1}{2} \\ \frac{|\Phi|^2 +1}{2} \end{pmatrix}  + \nabla \begin{pmatrix} \n \\ \langle \n , \Phi \rangle \\ \langle \n, \Phi \rangle \end{pmatrix}  \\
&=\nabla H \begin{pmatrix} \Phi \\ \frac{ |\Phi|^2 -1}{2} \\ \frac{|\Phi|^2 +1}{2} \end{pmatrix}  + H \begin{pmatrix} \nabla \Phi \\ \langle \nabla \Phi , \Phi \rangle \\ \langle \nabla \Phi, \Phi \rangle \end{pmatrix}  +  \begin{pmatrix} \nabla \n \\ \langle \nabla \n , \Phi \rangle \\ \langle\nabla \n, \Phi \rangle \end{pmatrix}  \\
&=\nabla H \begin{pmatrix} \Phi \\ \frac{ |\Phi|^2 -1}{2} \\ \frac{|\Phi|^2 +1}{2} \end{pmatrix}  -  \Ar \begin{pmatrix} \nabla \Phi \\ \langle \nabla \Phi , \Phi \rangle \\ \langle \nabla \Phi, \Phi \rangle \end{pmatrix}.
\end{aligned}\end{equation}
Hence :  
\begin{equation}
\label{lametriqueduY} \begin{aligned}\langle \partial_i Y, \partial_j Y \rangle &= \langle \Ar^p_i \partial_p \Phi, \Ar^q_j \partial_q \Phi \rangle \\
&= \Ar^p_i \Ar_{pj} = \Ar^p_i \Ar_p^l g_{jl} = \left( \Ar^T \Ar g \right)_{ij} \\
&= \frac{1}{2}|\Ar|^2 g_{ij} \text{ since } \Ar \text{ is symetric tracefree}.
\end{aligned}\end{equation}
Indeed since $\Ar$ is symetric tracefree it can be written $\Ar = \begin{pmatrix} \omega & \varphi \\ \varphi & - \omega \end{pmatrix}$ for $\omega, \varphi \in \R$ and thus $\Ar^T \Ar = \left( \omega^2 + \varphi^2 \right) Id$ while $ \left| \Ar \right| = 2 \left( \omega^2 + \varphi^2 \right)$. Hence 
\begin{equation}
\label{lauxiliaredelametriqueduY}
\Ar^p_i \Ar_p^l g_{jl} =  \frac{1}{2}\big|\Ar \big|^2 g_{ij}.
\end{equation}
We then deduce that $ Y \, : \,\left( \Sigma,g \right) \rightarrow \s^{4,1}$ is conformal. One may notice that the umbilic points of $\Phi$ are critical points of $Y$.
\end{ex}

As an enveloping spherical congruence, the conformal Gauss map carry many informations on the geometry of the immersion. Its key role is further emphasized by the fact it is the only conformal enveloping spherical congruence, up to orientation.

\begin{theo}
%[The only enveloping conformal spherical congruence]
\label{Yseulecongruenceenveloppante}
Let $\Sigma$ be a Riemann surface and $\Phi \, : \, \Sigma \rightarrow \R^3$ an immersion. We denote $g$ its first fundamental form and $Y$ its conformal Gauss map.
If the set of umbilic points of  $\Phi $ is nowhere dense then $Y$ and $-Y$ are the only smooth conformal $(\Sigma , g ) \rightarrow \R^3$ spherical congruences enveloping $\Phi$.
\end{theo}

\begin{proof}
As stated when we introduced it, the conformal Gauss map is a spherical congruence enveloping $\Phi$ which happens to be conformal.

 Conversely we consider a spherical congruence $G$ enveloping $\Phi$. %We denote $$p(\Phi)=\begin{pmatrix}  \Phi \\ \frac{|\Phi|^2+1}{2} \\ \frac{|\Phi|^2-1}{2} \end{pmatrix}.$$
 
%Let $E =  \mathrm{Vect} \left( \begin{pmatrix}  \Phi \\ \frac{|\Phi|^2+1}{2} \\ \frac{|\Phi|^2-1}{2} \end{pmatrix}, \partial_x \begin{pmatrix}  \Phi \\ \frac{|\Phi|^2+1}{2} \\ \frac{|\Phi|^2-1}{2} \end{pmatrix}, \partial_y \begin{pmatrix}  \Phi \\ \frac{|\Phi|^2+1}{2} \\ \frac{|\Phi|^2-1}{2} \end{pmatrix} \right)$.
Let $E =  \mathrm{Vect} \left(p(\Phi), \partial_x p(\Phi), \partial_y p(\Phi) \right)$.
 Equations (\ref{enveloppe1phi}) and (\ref{enveloppe2phi}) force $G$ to lie in $\left(E \right)^\perp$. Since $\Phi$ is an immersion, $E$ is of dimension $3$, and its orthogonal is then of dimension $2$. $Y$ envelopes $\Phi$  and $p(\Phi)$is isotropic, hence $\left( Y, p(\Phi)\right)$ is  a basis of $(E)^\perp$. $G$ can then be written as 
$$ G = \mu Y + \lambda p(\Phi)$$
with $\mu$, $\lambda \in \R$.

Since $\langle G, G \rangle = \mu^2 = 1$ one finds $\mu = \pm 1$ and deduce $\nabla \mu =0$.
We then need only to compute the first fundamental form of $G$ :
$$ \begin{aligned}\langle \partial_i G, \partial_j G \rangle &= \left\langle \mu \partial_i Y + \partial_i \lambda p(\Phi) + \lambda \begin{pmatrix} \partial_i \Phi \\\langle \partial_i \Phi , \Phi \rangle \\ \langle \partial_i \Phi , \Phi \rangle \end{pmatrix} ,  \mu \partial_j Y + \partial_j \lambda p(\Phi) + \lambda \begin{pmatrix} \partial_j \Phi \\\langle \partial_j \Phi , \Phi \rangle \\ \langle \partial_j \Phi , \Phi \rangle \end{pmatrix}\right\rangle  \\
&= \left\langle  \lambda \begin{pmatrix} \partial_i \Phi \\\langle \partial_i \Phi , \Phi \rangle \\ \langle \partial_i \Phi , \Phi \rangle \end{pmatrix} - \Ar^p_i \begin{pmatrix} \partial_p \Phi \\\langle \partial_p \Phi , \Phi \rangle \\ \langle \partial_ \Phi , \Phi \rangle \end{pmatrix},    \lambda \begin{pmatrix} \partial_j \Phi \\\langle \partial_j \Phi , \Phi \rangle \\ \langle \partial_j \Phi , \Phi \rangle \end{pmatrix} - \Ar^p_j \begin{pmatrix} \partial_p \Phi \\\langle \partial_p \Phi , \Phi \rangle \\ \langle \partial_ \Phi , \Phi \rangle \end{pmatrix} \right\rangle \end{aligned} $$
 using expression (\ref{lenabladuY}) of $\nabla Y$ and the fact that $p(\Phi) \in E^\perp$. Then 
$$\begin{aligned} \langle \partial_i G, \partial_j G \rangle &=  \langle \lambda \partial_i \Phi - \Ar^p_i \partial_p \Phi, \lambda \partial_j \Phi -\Ar^q_j \partial_q \Phi \rangle \\
&=\lambda^2 g_{ij} + \Ar^p_i \Ar_{pj} - 2 \lambda \Ar_{ij} \\
&= \left( \lambda^2 +\frac{|\Ar|^2}{2} \right)g_{ij}- 2 \lambda \Ar_{ij} 
\end{aligned}$$
where we have used (\ref{lauxiliaredelametriqueduY}). By hypothesis the set of umbilic points is nowhere dense,  $G$ is then conformal if and only if $\lambda =0$. We then have $G = \pm Y$ which concludes the proof. 
\end{proof}

Taking $-Y$ instead of $Y$ is tantamount to changing the orientation of the surface (taking $-\n$ instead of $\n$ as Gauss map).

Geometrically speaking $Y$ can be seen as the $2$ dimensional generalization of the osculating circles for curves in euclidian spaces, and it  will be of major importance in the study of Willmore surfaces, playing much of the same role as the Gauss map in the case of constant mean curvature surfaces.

\subsection{The conformal Gauss map in the three representations }
Since $Y$ conserves the conformal structure on $\Sigma$  it is convenient, and will not induce any loss of generality, to work in complex coordinates in local conformal charts. 
In the following we will then consider $\Phi \, : \, \D \rightarrow \R^3$ a smooth conformal immersion, that is satisfying $\left\langle \Phi_z, \Phi_z \right\rangle = 0$. Let $\n = \frac{\Phi_z \times \Phi_{\zb} }{ i \left| \Phi_z \right|^2 }$ denote its Gauss map with $\times$ the classic vectorial product in $\R^3$, $\lambda = \frac{1}{2}\log \left( 2 \left| \Phi_z \right|^2 \right)$ its conformal factor and $H = \left\langle \frac{ \Phi_{z \zb}}{ \left| \Phi_z \right|^2} , \n \right\rangle $ its mean curvature. Its tracefree curvature is defined as follows 
$$\Omega:= 2\left\langle \Phi_{zz}, \n \right\rangle. $$

\vspace{5mm}
Its representation in $\s^3$,  $X = \pi^{-1} \circ \Phi =\frac{1}{1 + \left| \Phi \right|^2} \begin{pmatrix} 2 \Phi  \\ \left| \Phi \right|^2 -1 \end{pmatrix}$ is conformal. Let  $\Lambda := \frac{1}{2} \log \left( 2 \left| X_z \right|^2 \right)$ be its conformal factor,  $\vec{N}$ such that $\left( X, e^{-\Lambda} X_x, e^{- \Lambda} X_y, \vec{N} \right)$ is a direct orthonormal basis of $\R^4$ its Gauss map ,  $ h = \left\langle  \frac{X_{z\zb}}{\left| X_z \right|^2}, \vec{N} \right\rangle$ its  mean curvature and $\omega := 2\left\langle X_{zz}, \vec{N} \right\rangle$ its tracefree curvature.
\vspace{5mm}
Similarly its representation in $\Hy^3$, $Z = \tilde \pi^{-1} \circ \Phi$ is conformal. Let  $\lambda^Z :=\frac{1}{2} \log \left( 2 \langle Z_z, Z_{\zb} \rangle \right)$ be its conformal factor,$ \n^Z$ such that $( Z, e^{-\lambda^Z} Z_x, e^{-\lambda^Z} Z_y, \n^Z)$ is a direct orthonormal basis of $\R^{3,1}$ its  Gauss map ,  $ H^Z = \left\langle \frac{Z_{z\zb}}{\left| Z_z \right|^2} , \n^Z \right\rangle$ its mean curvature and  $\Omega^Z := 2 \left\langle Z_{zz} , \n^Z \right\rangle$ its tracefree curvature.
One can then express $Y$ as the conformal Gauss map of an immersion in $\s^3$ or in $\Hy^3$.

\begin{prop}
\label{lapropositiondelexpressiondeY}
Let $\Phi$ be a smooth conformal immersion on $\D$, and $X$ (respectively $Z$) its representation in $\s^3$ (respectively $\mathbb{H}^3$) through $\pi$ (respectively $\tilde \pi$).  Let $Y$ be its conformal Gauss map. Then
$$ \begin{aligned}
Y &= h \begin{pmatrix} X \\ 1 \end{pmatrix} + \begin{pmatrix} \vec{N} \\ 0 \end{pmatrix} \\
&= H^Z \begin{pmatrix} Z_{h} \\-1 \\Z_4 \end{pmatrix} + \begin{pmatrix} \n^Z_{h} \\ 0 \\ \n^Z_4 \end{pmatrix}\end{aligned}$$
 where $Z = \begin{pmatrix} Z_h \\ Z_4 \end{pmatrix}$ and $\n^Z = \begin{pmatrix} \n^Z_h \\ \n^Z_{4} \end{pmatrix}$, while $h$ and $H^Z$ are the respective mean curvatures.
\end{prop}

\begin{proof}
Computations are done in appendix.
% For the expression of $Y$ in terms of immersions in $\s^3$ see section \ref{sectionannexcalculss3}, for the one in terms of  immersions in $\mathbb{H}^3$ see section  \ref{sectionannexcalculsH3}.

\end{proof}

\subsection{Conformally CMC immersions}
A quick study of proposition \ref{lapropositiondelexpressiondeY} and (\ref{Yphi}) reveals that the mean curvature in the three models can be written as a function of $Y$, with interesting geometric interpretations.
\begin{cor}
Let $\Phi$ be a smooth conformal immersion on $\D$, and $X$ (respectively $Z$) its representation in $\s^3$ (respectively $\mathbb{H}^3$) through $\pi$ (respectively $\tilde \pi$).  Let $Y$ be its conformal Gauss map. 
Then 
\begin{equation}
\label{HenfonctiondeY}
\begin{aligned}
H &= Y_5-Y_4, \\
h &= Y_5, \\
H^Z &=- Y_4. 
\end{aligned}
\end{equation}
\end{cor}

We denote $$v_s = \begin{pmatrix} 0\\0\\0\\1\\0\end{pmatrix}, \, v_t=\begin{pmatrix}0\\0\\0\\0\\1\end{pmatrix} \text{, } v_l = \begin{pmatrix}0\\0\\0\\1\\1\end{pmatrix}.$$
One deduces immediately from this that $\Phi$ is minimal (respectively of constant mean curvature) if and only if $Y_4 = Y_5$ (respectively if there exists a constant $H_0 \in \R$ such that $Y_5- Y_4 -H_0 = 0$), $X$ is minimal (respectively of constant mean curvature) if and only $Y_5= 0$ (respectively if there exists a constant $h_0 \in \R$ such that $Y_5-h_0= 0$), $Z$ is minimal (respectively of constant mean curvature)  if and only if $Y_4 =0$ (respectively if there exists a constant $H^Z_0 \in \R$ such that $Y_4+ H^Z_0= 0$). This can be reframed as : $\Phi$ is minimal (respectively CMC) if and only if $Y$ is in a linear (respectively affine) hyperplane of lightlike normal $v_l$,  $X$ is minimal (respectively CMC) if and only if $Y$ is in a linear (respectively affine) hyperplane of timelike normal $v_t$,  $Z$ is minimal (respectively CMC) if and only if $Y$ is in a linear (respectively affine) hyperplane of spacelike normal $v_s$.

We now dispose of a geometric characterization for the conformal Gauss maps of minimal surfaces in any of the three models. It is interesting to study how this condition, and thus $Y$, change under the action of conformal diffeomorphisms.

\begin{prop}
\label{modificationYtransformationconforme}
Let $\varphi \in \mathrm{Conf}( \s^3)$ corresponding to $M \in SO(4,1)$. 
Let $X \, : \, \Sigma \rightarrow \s^3$ be a smooth conformal immersion of conformal Gauss map $Y$.  We assume the  set of umbilic points of $X$ to be nowhere dense.
Let $Y_\varphi$ be the conformal Gauss map of $\varphi \circ X$.
Then $$Y_\varphi = M Y.$$
\end{prop}

\begin{proof}
We work in a conformal chart on a disk.
Thanks to theorem \ref{Yseulecongruenceenveloppante} one just needs to prove that $M Y$ is conformal, envelopes $\varphi \circ X$ and has the same orientation as  $Y_\varphi$.
% Indeed since both $\varphi \, \s^3 \rightarrow \s^3$ and $ X \, \left( \D, \langle .,. \rangle_2 \right) \rightarrow \left( \s^3, \langle .,. \rangle_4 \right)$ are conformal, showing that $Y_\varphi \, \left( \D, \langle .,. \rangle_2 \right) \rightarrow \left( \s^{4,1}, \langle .,. \rangle_{4,1} \right) $ is conformal shows that $Y_\varphi \, \left( \D, X^*\left( \langle .,. \rangle_4 \right) \right) \rightarrow \left( \s^{4,1}, \langle .,. \rangle_{4,1} \right) $ is conformal.

We first show that $MY$ is conformal.
Since $ \left( M Y \right)_z = M Y_z$ and $ M \in SO(4,1)$,
 $$ \begin{aligned}
  \langle \left( M Y\right)_z, \left( M Y\right)_z \rangle &= \langle Y_z, Y_z \rangle.
\end{aligned}$$
Given that $Y$ is conformal, one finds $ \langle\left( M Y \right)_z, \left( M Y \right)_z \rangle=0$, that is $MY$ is conformal.
We then justify that $MY$ envelopes $\varphi \circ X$.  To that aim we denote $V = M \begin{pmatrix} X \\ 1 \end{pmatrix} = \begin{pmatrix}  V_\circ \\ V_5 \end{pmatrix}$.
In accordance with corollary \ref{actionr3s3}, $\varphi( X) = \frac{V_\circ}{V_5 }$, which translates to 

\begin{equation}
\label{MchangeX} p(\varphi (X)) = \frac{1}{ V_5 } M p(X).
\end{equation}
Then 
$$ \begin{aligned} \left\langle M Y, p(\varphi(X)) \right\rangle &=   \frac{1}{ V_5 } \left\langle MY ,  Mp(X) \right\rangle \\
&=\frac{1}{ V_5 } \left\langle Y ,  p(X) \right\rangle \\
&= 0,
\end{aligned}$$
which proves  (\ref{enveloppe1X}), 
and 
$$ \begin{aligned} \left\langle M Y, \nabla p(\varphi(X)) \right\rangle &=  \nabla \left( \left\langle MY , p(\varphi (X)) \right\rangle \right) - \left\langle M \nabla Y,  p(\varphi(X)) \right\rangle_{4,1}  \\
&=- \frac{1}{V_5 } \left\langle M \nabla Y ,  M p(X) \right\rangle \\
&= - \frac{1}{V_5 } \left\langle  \nabla Y ,   p(X) \right\rangle \\
&= 0,
\end{aligned}$$
which shows  (\ref{enveloppe2X}) and that $ MY$ envelopes $\varphi(X)$.

Finally one  need only adress the orientation of $\varphi(X)$ to conclude. Let $N_\varphi$ be the Gauss map of $\varphi\circ X$ induced by the Gauss map $N$ of $X$, namely $N_\varphi = \frac{d\varphi (N)}{|d\varphi(N)|}$. Given the expression (\ref{projectionstereo6}) of the conformal Gauss map, $MY= Y_\varphi$ if and only if $\left\langle MY , \begin{pmatrix} N_\varphi \\ 0 \end{pmatrix} \right\rangle= 1$, $MY = - Y_\varphi$ otherwise. Let $W = M \begin{pmatrix} N \\ 0 \end{pmatrix}= \begin{pmatrix} W_\circ \\ W_5 \end{pmatrix}$.
With a straightforward computation one finds $$ d\varphi (N) = \frac{W_\circ}{V_5} - \frac{W_5}{V_5^2} V_\circ ,$$
which yields
$$ N_\varphi = W_\circ - \frac{W_5}{V_5} V_\circ.$$
Then $$\begin{aligned} \left\langle MY, \begin{pmatrix}N_\varphi \\ 0 \end{pmatrix} \right\rangle &= \left\langle MY , \begin{pmatrix}  W_\circ - \frac{W_5}{V_5} V_\circ \\ 0 \end{pmatrix} \right\rangle  \\
&=  \left\langle MY , \begin{pmatrix}  W_\circ \\W_5 \end{pmatrix}  -\begin{pmatrix} \frac{W_5}{V_5} V_\circ \\  W_5 \end{pmatrix} \right\rangle
 \\
&=  \left\langle MY ,M \begin{pmatrix}  N \\0 \end{pmatrix} - \frac{W_5}{V_5}\begin{pmatrix}  V_\circ \\  V_5 \end{pmatrix} \right\rangle \end{aligned}$$ 
thanks to the definition of $W$. Then due to $ \left\langle MY,M \begin{pmatrix}  N \\0 \end{pmatrix} \right\rangle = \left\langle Y, \begin{pmatrix} N \\ 0 \end{pmatrix} \right\rangle = 1$ one finds 

$$\begin{aligned} \left\langle MY, \begin{pmatrix}N_\varphi \\ 0 \end{pmatrix} \right\rangle &= 1 -\left\langle MY,  \frac{W_5}{V_5}\begin{pmatrix}  V_\circ \\  V_5 \end{pmatrix} \right\rangle \\
&= 1-\frac{W_5}{V_5} \left\langle MY, p(X) \right\rangle, \end{aligned}$$ 
by definition of  $V$. The equality $ \left\langle MY, Mp(X) \right\rangle = 0$  gives the expected result. 

Then $MY = Y_\varphi$ which is the desired result.

\end{proof}

One has similar results in the $\R^3$ and $\Hy^3$ settings.

\begin{prop}
Let $\varphi \in \mathrm{Conf}( \R^3 \cup \{ \infty \})$ corresponding to $M \in SO(4,1)$. 
Let $\Phi \in C^\infty (\Sigma, \R^3 )$ be a smooth immersion and $Y$ its conformal Gauss map. We assume the  set of umbilic points of $\Phi$ to be nowhere dense.
Let $Y_\varphi$ be the conformal Gauss map of $\varphi \circ \Phi$.
Then $$Y_\varphi = M Y.$$
\end{prop}

\begin{prop}
Let $\varphi \in \mathrm{Conf}( \Hy^3)$ corresponding to $M \in SO(4,1)$. 
Let $Z \in C^\infty (\Sigma , \Hy^3)$ be a smooth conformal immersion and $Y$ its conformal Gauss map. We assume the  set of umbilic points of $Z$ to be nowhere dense.
Let $Y_\varphi$ be the conformal Gauss map of $\varphi \circ Z$.
Then $$Y_\varphi = M Y.$$
\end{prop}

Then, since any $M \in SO(4,1)$ conserves  hyperplanes in $\R^{4,1}$ and the type of vectors we deduce  the following theorem.

\begin{theo}
\label{conformementCMC}
Let $\Phi \, : \, \D \rightarrow \R^3$ be a smooth conformal immersion, and $X$ (respectively $Z$) its representation in $\s^3$ (respectively $\mathbb{H}^3$) through $\pi$ (respectively $\tilde \pi$).  Let $Y$ be its conformal Gauss map. We assume the  set of umbilic points of $\Phi$ (or equivalently, see (\ref{projectionstereo5}) and (\ref{projectionhyper5}), $X$ or $Z$)  to be nowhere dense.

We say that $\Phi$ (respectively $X$, $Z$) is conformally CMC (respectively minimal) if and only if there exists a conformal diffeomorphism $\varphi$ of $\R^3 \cup \{ \infty \}$ (respectively $\s^3$, $\Hy^3$) such that $\varphi \circ \Phi$ (respectively $\varphi \circ X$, $\varphi \circ Z$) has constant mean curvature (respectively is minimal) in $\R^3$ (respectively $\s^3$, $\Hy^3$).

Then 

\begin{itemize}
\item
$\Phi$ is conformally CMC (respectively minimal) in $\R^3$ if and only if $Y$ lies in an affine (respectively linear) hyperplane of $\R^{4,1}$ with lightlike normal.
\item 
$X$  is conformally CMC (respectively minimal) in $\s^3$ if and only if $Y$ lies in an affine (respectively linear) hyperplane of $\R^{4,1}$ with timelike normal.
\item
$Z$ is conformally CMC (respectively minimal) in $\Hy^3$ if and only if $Y$ lies in an affine (respectively linear) hyperplane of $\R^{4,1}$ with spacelike normal.
\end{itemize}
\end{theo}

\subsection{Geometry of conformal Gauss maps }
\label{geometryofconformalGaussmaps}

Enveloping conditions (\ref{enveloppe1phi}) and (\ref{enveloppe2phi}) (or equivalently (\ref{enveloppe1X}) and (\ref{enveloppe2X}) or (\ref{enveloppe1Z}) and (\ref{enveloppe2Z})) ensure that $p(\Phi)$ (or equivalently   $p(X)$ or  $p(Z)$)  is an isotropic vector field normal to $Y$ in $\R^{4,1}$. 

We wish to complete $\left( Y , Y_z, Y_{\zb}, p(\Phi) \right)$ into a moving frame of $\R^{4,1}$  compatible with the decomposition $\R^{4,1} = TY \bigoplus NY$, in order to introduce the mean and tracefree curvatures of $Y$ as an immersion in $\R^{4,1}$.
As we pointed out prior, finding another immersion enveloped  by $Y$ is enough to complete the moving frame.

\begin{theo}
\label{dualsurfacephi}
Let $\Phi \, : \, \D \rightarrow \R^3$ be a smooth conformal immersion with no umbilic points. 
Then there exists $$\Phi^* = \Phi - \frac{4 H_z \overline{\Omega} e^{-2\lambda} }{ T(\Phi)}\Phi_z- \frac{4 H_{\zb} \Omega e^{-2\lambda} }{T(\Phi)} \Phi_{\zb} +  \frac{2H \left| \Omega \right|^2 e^{-2 \lambda} }{T(\Phi)} \n$$
where $T(\Phi) =  \left| \nabla H \right|^2 + H^2 \left| \Omega \right|^2 e^{-2\lambda}$, such that
\begin{equation} \label{hypphistar1} \left\langle Y ,p(\Phi^*) \right\rangle= 0\end{equation}
and
\begin{equation} \label{hypphistar2} \left\langle \nabla Y ,p(\Phi^*) \right\rangle = 0.\end{equation}
\end{theo}
\begin{proof}
We search for $\Phi^*$ under the form 
$$\Phi^* = \Phi + u \Phi_z + \overline{u} \Phi_{\zb} + v \n.$$
Applying first (\ref{hypphistar1}) then (\ref{hypphistar2})  yields
$$ \begin{aligned}v &= \frac{|u|^2 e^{2 \lambda} + v^2}{2} H \\
\Omega u &= - H_z \left( |u|^2 e^{2 \lambda} + v^2 \right).
\end{aligned}$$
Solving the resulting system gives us the desired values for $u$ and $v$.
\end{proof}

One can work similarly with immersions in $\s^3$.

\begin{theo}
\label{dualsurfaceX}
Let $X \, : \, \D \rightarrow \s^3$ be a smooth conformal immersion with no umbilic points. 
Then there exists $$X^* =\frac{ h^2 \left| \omega \right|^2 +4  \left| h_z \right|^2 e^{2 \Lambda} - \left| \omega\right|^2}{ T(X) } X - \frac{4 h_z \overline{\omega} }{T(X)}  X_z - \frac{4 h_{\zb} \omega}{ T(X) }  X_{\zb} + \frac{ 2 \left| \omega \right|^2 h}{T(X) } N$$
where $T(X) = { \left|\omega \right|^2 \left( 1+ h^2 \right) + 4 \left| h_z \right|^2 e^{2 \Lambda} } $, such that 
\begin{equation} \label{hypXstar1} \left\langle Y ,p(X^*) \right\rangle = 0\end{equation}
and
\begin{equation} \label{hypXstar2} \left\langle \nabla Y ,p(X^*) \right\rangle = 0.\end{equation}
\end{theo}

\begin{proof}
We search for $X^*$ under the form 
$$ X^* = \alpha X + \beta X_z + \beta X_{\zb} + \gamma N.$$
Applying first (\ref{hypXstar1}), then (\ref{hypXstar2})  yields 
$$ \begin{aligned}&\gamma = (1- \alpha) h,\\
&2h_z \left( \alpha-1 \right) = \omega \beta. \end{aligned}$$
Further $\left\langle X^*, X^* \right\rangle = 1$ ensures 
$$ \alpha^2 + \left| \beta \right|^2 e^{2\Lambda} + \gamma^2 = 1.$$
Solving the resulting system gives the desired result.
\end{proof}

%utiliser le noindent....

Let $e_{\Phi} :=\left( Y , Y_z, Y_{\zb}, p(\Phi),  p(\Phi^*) \right)$ and  $e_X:=\left( Y , Y_z, Y_{\zb}, p(X), p(X^*) \right)$ denote our two frames. Since $ p(\Phi)$ and $p(X)$ are colinear, necessarily $p(\Phi^*)$ and  $p(X^*)$ are too, meaning $X^* = \pi^{-1} \circ \Phi^*$, that is $X^*$ is the representation of $\Phi^*$ in $\s^3$.

Since $Y$ conformal, (\ref{hypphistar1}) and (\ref{hypphistar2}) (respectively (\ref{hypXstar1}) and (\ref{hypXstar2})), (\ref{enveloppe1phi}) and (\ref{enveloppe2phi})  (respectively (\ref{enveloppe1X}) and (\ref{enveloppe2X})) $e_\Phi$ (respectively $e_X$) is orthogonal.  For convenience's sake we will mainly work with $e_X$. Indeed while $\Phi$ is not necessarily contained in a compact, and thus neither is  $p(\Phi)$, $X \in \s^3$ makes for easier computations.  Each result has its counterpart in $\R^3$.

Let \begin{equation} \label{onfixenu} \nu=p(X)= \begin{pmatrix} X \\1 \end{pmatrix} ,\end{equation} $$l = \left\langle p(X), p(X^*) \right\rangle = \frac{-2 |\omega|^2}{|\omega|^2(h^2+1)+ \left|\nabla h \right|^2 e^{2\Lambda}}, $$ and
 \begin{equation} \label{onfixenustar} \begin{aligned} \nu^* &= -\frac{1}{l} p(X^*) \\
&= \frac{\left| \omega \right|^2 \left( h^2 +1 \right) + \left| \nabla h \right|^2 e^{2\Lambda} }{2 \left| \omega \right|^2} p(X^*) \\
&=   \begin{pmatrix} \left( \frac{h^2 -1}{2} + \frac{ \left| \nabla h \right|^2e^{2\Lambda}}{2 \left| \omega \right|^2} \right)  X  - \frac{2 h_z}{ \omega}   X_z - \frac{2 h_{\zb} }{ \overline{ \omega} }  X_{\zb} + h  N \\   \frac{h^2 +1}{2} + \frac{ \left| \nabla h \right|^2e^{2\Lambda}}{2 \left| \omega \right|^2}  \end{pmatrix}.
\end{aligned}\end{equation}
By design we have $\langle \nu, \nu^* \rangle =-1$.  Thus defined $ \left| \nu^*\right| < \infty$ away from umbilic points.

One computes easily, with Gauss-Codazzi (see (\ref{Gauss-Codazzis3}) in appendix) to obtain the second equality, 
\begin{equation}
\label{nuznustarbase}
\left\langle \nu_z  , \nu^* \right\rangle = -\frac{h_{\zb}e^{2\Lambda} }{ \overline{\omega} } = - \frac{ \overline{\omega}_z }{\overline{\omega} }.
\end{equation}
Using computations done in  (\ref{YzzbX}), one finds 
\begin{equation}
\label{Yzzbbase}
\begin{aligned}
Y_{z \zb} &=  \mathcal{W}_{\s^3}( X) \begin{pmatrix}X \\1 \end{pmatrix} - \frac{\left| \omega \right|^2 e^{-2\Lambda} }{2} Y = \mathcal{W}_{\s^3}( X) \nu - \frac{\left| \omega \right|^2 e^{-2\Lambda} }{2} Y
\end{aligned}
\end{equation}
where
$$\mathcal{W}_{\s^3}( X) = h_{z \zb} + \frac{ \left| \omega \right|^2 e^{-2\Lambda} }{2} h \in \R$$ as defined in (\ref{wtorduX}).
With the notations of section  \ref{formulas41}, see  (\ref{Yzzb}), this yields
\begin{equation}
\label{Hnubase}
H_\nu = 0,
\end{equation}
\begin{equation}
\label{e2ltordu}
e^{2 \mathcal{L}} =\left| \omega \right|^2 e^{-2\Lambda},
\end{equation}
and 
\begin{equation}
\label{Hnustarbase}
H_{\nu^*} = \frac{-2 \mathcal{W}_{\s^3} (X)}{ \left|\omega \right|^2 e^{-2 \Lambda} }.
\end{equation}

Similarly, applying (\ref{defomeganu}) to (\ref{YzzPhi}) we find 

\begin{equation}
\label{omeganubase}
\Omega_\nu = 2 \left\langle Y_{zz} , \nu \right\rangle = \omega,
\end{equation}
and
$$
\begin{aligned}
\Omega_{\nu^*} &= 2 \left\langle Y_{zz} , \nu^* \right\rangle \\
&=2 \left\langle  h_{zz} \begin{pmatrix} X \\ 1 \end{pmatrix} + h_z \begin{pmatrix} X_{z} \\ 0 \end{pmatrix} - \left( \omega e^{-2 \Lambda} \right)_z \begin{pmatrix}  X_{\zb} \\ 0 \end{pmatrix} - \omega \left(  \frac{h}{2} \begin{pmatrix}  \vec{N} \\ 0 \end{pmatrix} - \frac{1}{2} \begin{pmatrix} X \\ 0 \end{pmatrix} \right) , 
\nu^* \right\rangle \\
&=2 \left( h_{zz} + \frac{\omega}{2} \right) \left( \frac{h^2 -1}{2} + \frac{ \left| \nabla h \right|^2e^{2\Lambda}}{2 \left| \omega \right|^2} \right)  +  \frac{ h_z}{ \omega}  \left( \omega e^{-2\Lambda} \right)_z e^{2 \Lambda} - \frac{h_z h_{\zb} }{\overline{\omega} } e^{2\Lambda} - \frac{\omega h^2}{2} \\& - h_{zz} \left( \frac{h^2 +1}{2} + \frac{ \left| \nabla h \right|^2e^{2\Lambda}}{2 \left| \omega \right|^2} \right) \\
&= -\omega \frac{h^2 +1}{2} + 2\frac{ \left| \omega_{\zb} \right|^2 e^{-2\Lambda} }{ \overline{\omega}}+ 2\frac{ \omega_{\zb} \left( \omega e^{-2\Lambda} \right)_z }{\omega} -2\frac{ \omega_{\zb} \overline{\omega}_z e^{-2 \Lambda}}{ \overline{\omega} } -2 \left( \omega_{\zb} e^{-2 \Lambda} \right)_z  \\
&=   -\omega \frac{h^2 +1}{2} +  2\frac{ \omega_{\zb} \left( \omega e^{-2\Lambda} \right)_z }{\omega}  -2 \left( \omega_{\zb} e^{-2 \Lambda} \right)_z \\
&= 2\left( \frac{ \omega_{\zb}  \omega_z   }{\omega} - \omega_{z \zb} \right) e^{-2\Lambda}   -\omega \frac{h^2 +1}{2},
\end{aligned} $$
where we have used (\ref{Gauss-Codazzis3}) for the fourth equality. This yields
\begin{equation}
\label{omeganustarbase}
\begin{aligned}
\Omega_{\nu^*} &= -2 \omega e^{-2 \Lambda}  \left( \left( \frac{ \omega_{\zb} }{\omega}\right)_z + \frac{h^2 +1}{4} e^{2 \Lambda} \right) \\
&= - 2 \omega e^{-2 \Lambda}  \left( \left( \frac{ \omega_{z} }{\omega}\right)_{\zb} + \frac{h^2 +1}{4} e^{2 \Lambda} \right) .
\end{aligned}
\end{equation}

A consequence of these computations is that the conformal Gauss map of an immersion $X$ is necessarily of vanishing mean curvature in the direction $p(X)$. This is in fact an equivalence.

\begin{theo}
\label{conditionintegrabilite}
Let $Y \, : \, \D \rightarrow \s^{4,1}$ be a spacelike ( that  is $\left\langle Y_z , Y_{\zb} \right\rangle >0$) conformal immersion. Then $ Y$ is the conformal Gauss map of $X \, : \, \D \rightarrow \s^3$ if and only if there exists an isotropic normal direction $\nu$ such that $H_\nu =0$, where $H_\nu$ is the mean curvature in the $\nu$ direction defined in (\ref{defomeganu}). Moreover, $ \nu$ is parallel to $p(X)$.
\end{theo}

\begin{proof}
We have shown in (\ref{Hnubase}) that if $Y$ is the conformal Gauss map of $X$ then $Y$ is of null mean curvature in the isotropic $p(X)$ direction.

Reciprocally consider $Y$ of null mean curvature in the isotropic direction $\nu$ Let us build $X \, : \, \D  \rightarrow \s^3$ such that $Y$ is the conformal Gauss map of $X$. Since $\langle \nu , \nu \rangle = 0$ and $\nu \neq 0$, the last coordinate $\nu_5$ of $\nu$ is necessarily non null. One can then renormalize $\nu$ to $ \frac{\nu}{\nu_5} =p(X)$.  There then exists $ X \, : \D \rightarrow \s^3$ such that 
$$\begin{aligned}
&\left\langle Y , p(X) \right\rangle = 0, \\
&\left\langle Y_z , p(X) \right\rangle = 0, \\
&\left\langle Y_{z \zb} , p(X) \right\rangle = 0 .
\end{aligned}$$
One checks that hypotheses (\ref{enveloppe1X}) and (\ref{enveloppe2X}) are satisfied and that $Y$ envelopes $X$. We now just have to prove that $X$ is conformal and apply \ref{Yseulecongruenceenveloppante} to conclude.

Since $\left\langle X_z , X_z \right\rangle = \left\langle p(X)_z , p(X)_z \right\rangle$  and according to (\ref{nunustar})
$$  \left\langle p(X)_z ,p(X)_z \right\rangle = H_{p(X)} \Omega_{p(X)} = 0,$$
$X$ is shown to be conformal, which concludes the proof.
\end{proof}

We must draw the reader's attention to the fact that $Y$ is \emph{not} a priori the conformal Gauss map of $X^*$. Indeed while $Y$ envelopes $X^*$, $X^*$ is not necessarily conformal :
$$\begin{aligned}
\left\langle X^*_z , X^*_z \right\rangle &= \left\langle p( X^* )_z,p(X^*)_z \right\rangle \\
&=\left\langle \left( l \nu^*  \right)_z, \left(l \nu^*  \right)_z \right\rangle \\
&=  l^2 \left\langle \nu^*_z , \nu^*_z \right\rangle
\end{aligned}$$
since (\ref{nustarznustar}) stands and $\nu^*$ is isotropic. Then using (\ref{nunustar}) 

$$ \begin{aligned}
\left\langle X^*_z , X^*_z \right\rangle &=l^2 H_{\nu^*} \Omega_{\nu^*} \\ 
&=l^2\omega e^{-2 \Lambda}  \left( \left( \frac{ \omega_{z} }{\omega}\right)_{\zb} + \frac{h^2 +1}{4} e^{2 \Lambda} \right) \frac{2 \mathcal{W}_{\s^3} (X)}{ \left|\omega \right|^2 e^{-2 \Lambda} },
\end{aligned}$$
with (\ref{omeganustarbase}) and (\ref{Hnustarbase}).

Then 
\begin{equation}
\label{Xstarconforme?}
\begin{aligned}
\left\langle X^*_z , X^*_z \right\rangle &=\frac{4 \omega |\omega|^2 }{\left( |\omega|^2 \left( h^2+1 \right) + \left|\nabla h \right|^2 e^{2 \Lambda} \right)^2 }\mathcal{W}_{\s^3} (X) \left( \left( \frac{\omega_z}{\omega} \right)_{\zb} + \frac{h^2+1}{4} e^{2\Lambda} \right).
\end{aligned}
\end{equation}
One can notice that a simple condition to ensure that is $X^*$ is conformal is $\mathcal{W}_{\s^3} (X) =0$, that is $X$ is a Willmore immersion.
The computations for an immersion $\Phi$ in $\R^3$ (see (\ref{YzPhi})-(\ref{QPhi}))  bring to the forefront the quantity $$W( \Phi) = H_{z \zb} + \frac{ \left| \Omega \right|^2 e^{-2\lambda} }{2} H \in \R .$$ We refer the reader to (\ref{wtorduXPhi}) for the proof that $$\mathcal{W}_{\s^3} (X) = \frac{ |\Phi|^2 +1}{2} \mathcal{W}(\Phi ).$$

%We compute :

%$$\left\langle \begin{pmatrix} \Phi^* \\ \frac{| \Phi^*|^2 -1}{2} \\ \frac{| \Phi^*|^2+1}{2} \end{pmatrix}, \begin{pmatrix} \Phi \\ \frac{| \Phi|^2 -1}{2} \\ \frac{| \Phi|^2+1}{2} \end{pmatrix} \right\rangle =\frac{-2 \left| \Omega \right|^2 e^{-2 \lambda} }{ \left| \nabla H \right|^2 + H^2 \left| \Omega \right|^2 e^{-2 \lambda} } $$

%Applying computations in subsection \ref{formulas41}, namely (\ref{nunustar}), (\ref{defomeganustar})  and (\ref{defHnustar})  with $\nu_0 = \begin{pmatrix} \Phi \\ \frac{| \Phi|^2 -1}{2} \\ \frac{| \Phi|^2+1}{2} \end{pmatrix}$ and $\nu_0^* = \begin{pmatrix} \Phi^* \\ \frac{| \Phi^*|^2 -1}{2} \\ \frac{| \Phi^*|^2+1}{2} \end{pmatrix}$ we find 

%$$\begin{aligned}
%\Omega_{\nu_0^*} &= 2\left\langle Y_{z z} , \begin{pmatrix} \Phi^* \\ \frac{| \Phi^*|^2 -1}{2} \\ \frac{| \Phi^*|^2+1}{2} \end{pmatrix} \right\rangle_{4,1} \\
%&= 2 H_{zz} \left\langle \begin{pmatrix} \Phi \\ \frac{| \Phi|^2 -1}{2} \\ \frac{| \Phi|^2+1}{2} \end{pmatrix} ,\begin{pmatrix} \Phi^* \\ \frac{| \Phi^*|^2 -1}{2} \\ \frac{| \Phi^*|^2+1}{2} \end{pmatrix} \right\rangle_{4,1} + H_z \left\langle \begin{pmatrix} \Phi_z \\ \langle \Phi_z , \Phi \rangle_3 \\ langle \Phi_z , \Phi \rangle_3  \end{pmatrix} ,\begin{pmatrix} \Phi^* \\ \frac{| \Phi^*|^2 -1}{2} \\ \frac{| \Phi^*|^2+1}{2} \end{pmatrix} \right\rangle_{4,1} - \left(\Omega
%\end{aligned}$$

\section{Conformal Gauss map of Willmore Immersions}
As the quantity $\mathcal{W}$  appears in several computations linked to the geometry of $Y$, it is natural to  study the conformal Gauss map of immersions satisfying $\mathcal{W} = 0$ i.e. Willmore immersions.
\subsection{Willmore immersions }
We first recall the definition of Willmore surfaces.  
\begin{de}
Let $\Phi \, : \, \D \rightarrow \R^3$ be a conformal immersion of representation $X$ in $\s^3$ and $Z$ in $\Hy^3$. $\Phi$, $X$ and $Z$ are said to be Willmore immersions if $\mathcal{W} \left( \Phi \right) = 0$ (or equivalently, see (\ref{wtorduXPhi}), $\mathcal{W}_{\s^3} ( X) =0$).
\end{de}

In his studies of Willmore immersions, T. Rivière brought to light equations in divergence form satisfied by Willmore immersions (see  (7.15), (7.16) and (7.30c) in \cite{bibpcmi}) and the conserved quantities associated. Later Y. Bernard showed in \cite{bibnoetherwill}  they could be seen as consequences of the invariance of the Willmore functional $W(\Phi) = \int_{\D} H^2 e^{2\lambda} dz$ under the action of the conformal group.

\begin{theo}[(7.15), (7.16) and (7.30c) in \cite{bibpcmi}]
\label{thmV}
Let $\Phi \, : \,  \D \rightarrow \R^3$ be a conformal Willmore immersion. Then 
\begin{equation}
\label{equationsdivergencewillmore}
\begin{aligned}
&div \left(  \nabla \vec{H} - 3 \pi_{\n} \left( \nabla \vec{H} \right) + \nabla^\perp \n \times \vec{H} \right)  = 0, \\
&div \left( \left\langle\nabla \vec{H} - 3 \pi_{\n} \left( \nabla \vec{H} \right) + \nabla^\perp \n \times \vec{H}  , \Phi \right\rangle \right) = 0 \\
&div \left( \Phi \times \left( \nabla \vec{H} - 3 \pi_{\n} \left( \nabla \vec{H} \right) + \nabla^\perp \n \times \vec{H}  \right) + 2 H \nabla^\perp \Phi \right)=0 \\
&div \left(|\Phi|^2 \left(  \nabla \vec{H} - 3 \pi_{\n} \left( \nabla \vec{H} \right) + \nabla^\perp \n \times \vec{H}  \right) -2 \left\langle  \nabla \vec{H} - 3 \pi_{\n} \left( \nabla \vec{H} \right) + \nabla^\perp \n \times \vec{H} , \Phi \right\rangle \Phi \right. \\ & \left.+4 \Phi \times \left( \n \times \Ar \nabla \Phi \right)  \right) =0.
\end{aligned}
\end{equation}

This allows us to define the conserved quantities of $\Phi$ : 
\begin{equation}
\label{quantitesconserveewillmore}
\begin{aligned}
& V_{\mathrm{tra}} = \nabla \vec{H} - 3 \pi_{\n} \left( \nabla \vec{H} \right) + \nabla^\perp \n \times \vec{H}  \\
& V_{\mathrm{dil}} = \left\langle  \nabla \vec{H} - 3 \pi_{\n} \left( \nabla \vec{H} \right) + \nabla^\perp \n \times \vec{H}  , \Phi \right\rangle  = \left\langle V_{\mathrm{tra}}, \Phi \right\rangle \\
& V_{\mathrm{rot}} = \Phi \times \left( \nabla \vec{H} - 3 \pi_{\n} \left( \nabla \vec{H} \right) + \nabla^\perp \n \times \vec{H}  \right) + 2 H \nabla^\perp \Phi = \Phi \times  V_{\mathrm{tra}}  + 2 H \nabla^\perp \Phi \\
& V_{\mathrm{inv}} = -|\Phi|^2 \left( \nabla \vec{H} - 3 \pi_{\n} \left( \nabla \vec{H} \right) + \nabla^\perp \n \times \vec{H}  \right) \\&+2 \left\langle  \nabla \vec{H} - 3 \pi_{\n} \left( \nabla \vec{H} \right) + \nabla^\perp \n \times \vec{H}  , \Phi \right\rangle \Phi  -4 \Phi \times \left( \n \times \Ar \nabla \Phi \right)  \\
& \quad \quad = -|\Phi|^2 V_{\mathrm{tra}} +2V_{\mathrm{dil}} \Phi -4 \Phi \times \left( \n \times \Ar \nabla \Phi \right).
\end{aligned}
\end{equation}
\end{theo}

\begin{remark}
As suggested by the terminology $V_{\mathrm{tra}}$ follows from the invariance by translations, $V_{\mathrm{dil}}$ the invariance by dilations, $V_{\mathrm{rot}}$ the invariance by rotations and $V_{\mathrm{inv}}$  the invariance by transformations of the form $x \mapsto \frac{x-\vec{a}}{\left| x- \vec{a} \right|^2}$.
\end{remark}

While $ V_{\mathrm{tra}} = \nabla \vec{H} - 3 \pi_{\n} \left( \nabla \vec{H} \right) + \nabla^\perp \n \times \vec{H} $ is more apt  to higher codimensions generalizations, we will prefer another expression. Using (\ref{formuleennablaperpn}) one has

\begin{equation}
\label{vtranslalter}
\begin{aligned}
V_{\mathrm{tra}} &=  \nabla \vec{H} - 3 \pi_{\n} \left( \nabla \vec{H} \right) + \nabla^\perp \n \times \vec{H}  \\
&=-2 \nabla H \n + H \nabla \n + H^2 \nabla \Phi - H \Ar \nabla \Phi \\
&= -2 \left( \nabla H \n + H \Ar \nabla \Phi \right).
\end{aligned}
\end{equation}

There is a  notion closely linked to Willmore immersions which will be useful later, called conformal Willmore immersions.
\begin{de}
%[Conformal Willmore, theorem X.18 in \cite{bibconservationlawsforconformallyinvariantproblems}]
\label{conformalWillmoreimmersion}
Let $\Phi \, : \, \D \rightarrow \R^3$ be a conformal immersion of representation $X$ in $\s^3$ and $Z$ in $\Hy^3$. $\Phi$, $X$ and $Z$ are said to be conformal Willmore immersions if there exists an holomorphic function $F$ such that $\mathcal{W} \left( \Phi \right) = \Re \left( \overline{F} \Omega e^{-2\lambda} \right)$ (or equivalently, see (\ref{wtorduXPhi}), an holomorphic function $f$ such that $\mathcal{W}_{\s^3} ( X) = \Re \left( \overline{f} \omega e^{-2 \Lambda} \right)$).
\end{de}

While Willmore immersions are critical points of the Willmore functional, conformal Willmore immersions are critical points of the Willmore functional \emph{in a conformal class} and $f$ acts as a Lagrange multiplier (see subsection X.7.4 in \cite{bibconservationlawsforconformallyinvariantproblems} for more details).

\subsection{Willmore and harmonic conserved quantities}
Equality (\ref{YzzbPhi}) (or equivalently (\ref{YzzbX})) yields the following theorem.

\begin{theo}
Let $\Phi \, : \, \D \rightarrow \R^3$ be a conformal immersion of representation $X$ in $\s^3$ and $Z$ in $\Hy^3$.
Then $\Phi$ is Willmore if and only if its conformal Gauss map $Y$ is minimal, that is if it is conformal and satisfies 
$$Y_{z\zb} + \left\langle Y_z , Y_{\zb} \right\rangle  Y = 0$$
which in real notations is tantamount to
\begin{equation}
\label{equationYminimale}
\Delta Y + \left\langle \nabla Y  . \nabla Y \right\rangle Y =0.
\end{equation}
\end{theo}

Then  assuming (\ref{equationYminimale}), for all $i,j \in \{1 \dots 5\} $ 

$$div \left( \nabla Y_i Y_j - Y_i \nabla Y_j \right) = \Delta Y_i Y_j - \Delta Y_j Y_i = 0.$$
$Y$ then satisfies the following conservation laws (that can actually be thought to follow from the invariance group $SO(4,1)$ of the energy $E(Y) = \int_{\D} \left\langle \nabla Y. \nabla Y \right\rangle dz$) : 

\begin{equation}
\label{conservationlawsY}
div \left( \nabla Y Y^T - Y \nabla Y^T \right) = 0.
\end{equation}

These conservation laws stem from the seminal works of F. Hélein on harmonic maps in the euclidean spheres (see \cite{bibharmmaps} for an extensive study) and the generalization of M. Zhu to harmonic maps in de Sitter spaces in \cite{MR3003285}.

\begin{theo}
\label{repartitionquantiteconservees}
Let $\Phi \, : \, \D \rightarrow \R^3$ be a Willmore immersion, conformal,  of conformal Gauss map $Y$. 
Let 
$$\begin{aligned} 
\mu = \begin{pmatrix}\nabla Y_i Y_j - Y_i \nabla Y_j \end{pmatrix} = \nabla Y Y^T - Y \nabla Y^T. \end{aligned}$$
Then 
\sbox0{$\begin{matrix}0 & - \tilde V_{\mathrm{rot} \,3 } &  \tilde V_{\mathrm{rot} \, 2}\\\tilde V_{\mathrm{rot} \, 3} & 0 & - \tilde V_{\mathrm{rot} \, 1}\\- \tilde V_{\mathrm{rot} \, 2} & \tilde V_{\mathrm{rot} \, 1} & 0\end{matrix}$}

$$2 \mu = \begin{pmatrix} U &-\frac{V_{\mathrm{tra}} - V_{\mathrm{inv}}}{2} & \frac{ V_{\mathrm{tra}} + V_{\mathrm{inv}} }{2} \\ \left(\frac{V_{\mathrm{tra}} - V_{\mathrm{inv}}}{2} \right)^T & 0 & V_{\mathrm{dil}} \\ -\left(\frac{V_{\mathrm{inv}}+ V_{\mathrm{tra}} }{2} \right)^T & -V_{\mathrm{dil}} & 0 \end{pmatrix}$$
%$$2 \mu=\left(
%\begin{array}{ccc}
%U &\makebox[\wd0]{\large $ -\frac{V_{\mathrm{tra}} - V_{\mathrm{inv}}}{2}$}& \makebox[\wd0]{ \large $\frac{ V_{\mathrm{tra}} + V_{\mathrm{inv}} }{2}$}\\
%  \vphantom{\usebox{0}}\makebox[\wd0]{\large $\left(\frac{V_{\mathrm{tra}} - V_{\mathrm{inv}}}{2} \right)^T$}&\makebox[\wd0]{0}&\makebox[\wd0]{$V_{\mathrm{dil}}$} \\
%  \vphantom{\usebox{0}}\makebox[\wd0]{\large $-\left(\frac{V_{\mathrm{inv}}+ V_{\mathrm{tra}} }{2} \right)^T$}&\makebox[\wd0]{$-V_{\mathrm{dil}}$}&\makebox[\wd0]{0}
%\end{array}
%\right)$$
where $V_{\mathrm{tra}}, V_{\mathrm{dil}}, V_{\mathrm{rot}} $ and $V_{\mathrm{inv}}$ are defined in theorem \ref{equationsdivergencewillmore} and $$ U = \left(\begin{matrix}0 & - \tilde V_{\mathrm{rot} \,3 } &  \tilde V_{\mathrm{rot} \, 2}\\\tilde V_{\mathrm{rot} \, 3} & 0 & - \tilde V_{\mathrm{rot} \, 1}\\- \tilde V_{\mathrm{rot} \, 2} & \tilde V_{\mathrm{rot} \, 1} & 0\end{matrix} \right)$$ with $ \tilde  V_{\mathrm{rot}} = V_{\mathrm{rot}} + 2 \nabla^\perp \n.$
\end{theo}
\begin{proof}

We decompose $\mu$ in blocks : 
$$\mu = \begin{pmatrix} P & a & b \\-a^T & 0 & \omega \\ -b^T & -\omega &0 \end{pmatrix},$$ with $P \in M_3( \R)$ antisymetric, $a,b \in \R^3$ and $\omega \in \R$.
Let $ \epsilon = \begin{pmatrix} 1&0&0&0&0 \\
0&1&0&0&0 \\0&0&1&0&0 \\0&0&0&1&0 \\0&0&0&0&-1 \end{pmatrix}$.
Then given any $a,b \in \R^5$, $$a^T \epsilon b = \langle a,b \rangle$$ where $\langle .,.\rangle$ is the Lorentzian product in $\R^{4,1}$.

For any $w \in \R^3$
$$ \begin{aligned} \mu \epsilon \begin{pmatrix} w \\ 0 \\0 \end{pmatrix} &= \left\langle Y,\begin{pmatrix}w \\0 \\ 0 \end{pmatrix} \right\rangle \nabla Y - \left\langle \nabla Y , \begin{pmatrix}w \\0 \\ 0 \end{pmatrix} \right\rangle Y \\
&=\left\langle H\Phi + \n , w \right\rangle \left[ \nabla H \begin{pmatrix} \Phi \\ \frac{|\Phi|^2-1}{2} \\ \frac{|\Phi|^2 +1}{2} \end{pmatrix} - \Ar \begin{pmatrix} \nabla \Phi \\ \langle \nabla \Phi, \Phi \rangle \\ \langle \nabla \Phi , \Phi \rangle \end{pmatrix} \right]  \\
&-\left\langle \nabla H \Phi - \Ar \nabla \Phi , w \right\rangle \left[  H \begin{pmatrix} \Phi \\ \frac{|\Phi|^2-1}{2} \\ \frac{|\Phi|^2 +1}{2} \end{pmatrix} + \begin{pmatrix} \n \\ \langle \n, \Phi \rangle \\ \langle \n , \Phi \rangle \end{pmatrix} \right] , \\ 
\end{aligned}$$
while 
$$ \mu \epsilon \begin{pmatrix} w \\ 0 \\0 \end{pmatrix} = \begin{pmatrix} Pw \\ - \langle a ,w \rangle \\ -\langle b ,w \rangle \end{pmatrix}.$$
Focusing on the first three coordinates yields 
$$ \begin{aligned} Pw &= \left\langle H\Phi + \n , w \right\rangle \left[ \nabla H  \Phi  - \Ar  \nabla \Phi \right]  -\left\langle \nabla H \Phi - \Ar \nabla \Phi , w \right\rangle \left[  H  \Phi +  \n  \right] \\
&= w \times \left[ \Phi \times \left( \nabla H \n + H \Ar \nabla \Phi \right) + \n \times \Ar \nabla \Phi \right]  \\
&= -\frac{1}{2} w \times \left[ \Phi \times V_{\mathrm{tra}} + 2 \Ar \nabla \Phi \times \n \right]= -\frac{1}{2} w \times \tilde V_{\mathrm{rot}}
\end{aligned} $$
with $\tilde V_{\mathrm{rot}} = V_{\mathrm{rot}}-2 \nabla^\perp \n =  \Phi \times V_{\mathrm{tra}} + 2 \Ar \nabla \Phi \times \n $. With this valid for all $w \in \R^3$ we deduce 
$$P =\frac{1}{2}\begin{pmatrix}0 & - \tilde V_{\mathrm{rot} \,3 } &  \tilde V_{\mathrm{rot} \, 2}\\\tilde V_{\mathrm{rot} \, 3} & 0 & - \tilde V_{\mathrm{rot} \, 1}\\- \tilde V_{\mathrm{rot} \, 2} & \tilde V_{\mathrm{rot} \, 1} & 0\end{pmatrix}.$$

 Similarly : 
$$ \begin{aligned} \mu \epsilon \begin{pmatrix} 0 \\ 1 \\1 \end{pmatrix} &= \left\langle Y,\begin{pmatrix}0 \\1 \\ 1 \end{pmatrix} \right\rangle \nabla Y - \left\langle \nabla Y , \begin{pmatrix}0\\1 \\ 1 \end{pmatrix} \right\rangle Y \\
&=-H \left[ \nabla H \begin{pmatrix} \Phi \\ \frac{|\Phi|^2-1}{2} \\ \frac{|\Phi|^2 +1}{2} \end{pmatrix} - \Ar \begin{pmatrix} \nabla \Phi \\ \langle \nabla \Phi, \Phi \rangle \\ \langle \nabla \Phi , \Phi \rangle \end{pmatrix} \right]  +\nabla H  \left[  H \begin{pmatrix} \Phi \\ \frac{|\Phi|^2-1}{2} \\ \frac{|\Phi|^2 +1}{2} \end{pmatrix} + \begin{pmatrix} \n \\ \langle \n, \Phi \rangle \\ \langle \n , \Phi \rangle \end{pmatrix} \right]  \\ 
&= \nabla  H \begin{pmatrix} \n \\ \langle \n, \Phi \rangle \\ \langle \n , \Phi \rangle \end{pmatrix} + H \Ar \begin{pmatrix} \nabla \Phi \\ \langle \nabla \Phi , \Phi \rangle \\ \langle \nabla \Phi , \Phi \rangle \end{pmatrix},
\end{aligned}$$
while  $$\mu \epsilon \begin{pmatrix} 0 \\ 1 \\1 \end{pmatrix} = \begin{pmatrix} a-b \\ -\omega \\ -\omega \end{pmatrix}.$$
Hence $$ \begin{aligned} a-b &= - \frac{V_{\mathrm{tra}}}{2}, \\
\omega &= \frac{V_{\mathrm{dil}}}{2}. \end{aligned}$$

In an alike manner, computing in two ways $\mu \epsilon \begin{pmatrix}0 \\ 1 \\-1 \end{pmatrix} $ yields 
$$a+b = \frac{V_{\mathrm{inv}}}{2}.$$
Hence $$\begin{aligned} a&= -\frac{V_{\mathrm{tra}} - V_{\mathrm{inv}}}{4} \\
b&=  \frac{V_{\mathrm{inv}} +V_{\mathrm{tra}}}{4}. \end{aligned}$$

To conclude we assemble all the previous results and reach

\sbox0{$\begin{matrix}0 & - \tilde V_{\mathrm{rot} \,3 } &  \tilde V_{\mathrm{rot} \, 2}\\\tilde V_{\mathrm{rot} \, 3} & 0 & - \tilde V_{\mathrm{rot} \, 1}\\- \tilde V_{\mathrm{rot} \, 2} & \tilde V_{\mathrm{rot} \, 1} & 0\end{matrix}$}

$$2 \mu = \begin{pmatrix} U &-\frac{V_{\mathrm{tra}} - V_{\mathrm{inv}}}{2} & \frac{ V_{\mathrm{tra}} + V_{\mathrm{inv}} }{2} \\ \left(\frac{V_{\mathrm{tra}} - V_{\mathrm{inv}}}{2} \right)^T & 0 & V_{\mathrm{dil}} \\ -\left(\frac{V_{\mathrm{inv}}+ V_{\mathrm{tra}} }{2} \right)^T & -V_{\mathrm{dil}} & 0 \end{pmatrix}$$
which is the desired result.

\end{proof}
\begin{remark}
While $V_{\mathrm{rot}}$ stems from the invariance by rotation of the Willmore energy $W(\Phi) = \int H^2 e^{2\lambda}dz^2$, $\tilde V_{\mathrm{rot}}$ is a consequence of the invariance by rotation of  $\int \left| \Omega \right|^2 e^{-2\lambda}dz^2$. These two functionals differ by a topological invariant and thus have the same critical points, with the same set of conserved quantities.  However one might favor the second one since $\left| \Omega \right|^2 e^{-2\lambda}dz^2$ is a pointwise conformal invariant (unlike  $ H^2 e^{2\lambda}dz^2$).
\end{remark}

One of the advantages of this formulation is that it describes conveniently how these conserved quantities change under the action of diffeomorphisms.

\begin{theo}
\label{exchangelawgeneral}
Let $\Phi \, : \, \D \rightarrow \R^3$ be a Willmore immersion, conformal, of conformal Gauss map $Y$. Let $\mu$ be as in theorem \ref{repartitionquantiteconservees}. 
Let $\varphi \in \mathrm{Conf}\left( \R^3 \cup \{ \infty \} \right)$ and $M \in SO( 4,1)$ associated. Let $Y_\varphi$ be its conformal Gauss map and $\mu_\varphi$ be as in theorem \ref{repartitionquantiteconservees}.
Then $$\mu_\varphi = M\mu M^T.$$
\end{theo}
\begin{proof}
Using proposition \ref{modificationYtransformationconforme} one has $Y_\varphi = M Y$ and since
 $$\mu_{\varphi} =  Y_\varphi \left( \nabla Y_\varphi \right)^T - \nabla Y_\varphi \left( Y_\varphi \right)^T = M \left( Y \nabla Y^T - \nabla Y Y^T \right) M^T=  M \mu M^T.$$
\end{proof}

As an example theorem \ref{exchangelawgeneral} yields an alternative proof of a result by A. Michelat and T. Rivière in \cite{michelatclassifi} that describes the exchange laws of conserved quantities under the action of the inversion at the origin.

\begin{cor}
\label{laloidechange}
Let $\Phi \, : \, \D \rightarrow \R^3$ be a Willmore immersion, conformal, of conformal Gauss map $Y$.
Let $\iota \, : \, x \mapsto \frac{x}{|x|^2}$ be the inversion at the origin. Let $V_{*, \iota}$ be the conserved quantity corresponding to the transformation $*$ for $\iota \circ \Phi$.
Then 
$$ \begin{aligned}
V_{\mathrm{tra}, \, \iota } &= V_{\mathrm{inv}} \\
V_{\mathrm{inv}, \, \iota } &= V_{\mathrm{tra}} \\
V_{\mathrm{dil}, \, \iota} &= - V_{\mathrm{dil}} \\
\tilde V_{\mathrm{rot}, \, \iota} &= \tilde V_{\mathrm{rot}}.
\end{aligned}$$
\end{cor}
\begin{proof}

One need only apply theorem \ref{exchangelawgeneral} with $\varphi = \iota$ and $M = M_\iota = \begin{pmatrix} - Id &0&0 \\0&1&0 \\ 0&0&-1 \end{pmatrix}$ (see (\ref{reprinversion})), and interpret the result with theorem \ref{repartitionquantiteconservees}.

\end{proof}

On non simply-connected domains, each conserved quantity yields a corresponding residual which follow the exchange law presented in corollary \ref{laloidechange}. The exchange law of residuals was in fact the result obtained by A. Michelat and T. Rivière in \cite{michelatclassifi} and served as a key stepping stone for their classification of branched Willmore spheres.

\subsection{Conformal dual immersion }
As was pointed out in conclusion of subsection \ref{geometryofconformalGaussmaps}, a sufficient condition for $X^*$ to be conformal is $X$ Willmore. In that case $Y$ is the conformal Gauss map of $X^*$.

\begin{theo}
Let $X \, : \, \D \, \rightarrow  \s^3$ be a Willmore immersion, conformal, of conformal Gauss map $Y$. 
Then there exists a branched conformal Willmore immersion  $X^* \, : \,  \D \, \rightarrow \s^3$ such that $Y$ is the conformal Gauss map  of $X^*$. Then
$X^*$ is called the conformal dual immersion of $X$.
\end{theo}

\begin{proof}
Taking $X^*$ as in theorem \ref{dualsurfaceX}, and recalling (\ref{Xstarconforme?}) with $X$ Willmore, one finds $X^*$ conformal and enveloped by $Y$. Theorem \ref{Yseulecongruenceenveloppante} concludes.

\end{proof}

Another way to see this is to understand that $Y$ minimal means there are two isotropic directions in which $Y$ has zero mean curvature, meaning $Y$ is the conformal Gauss map of two immersions, according to theorem \ref{conditionintegrabilite}. One is $X$, the other is its conformal dual.

\subsection{Bryant's functional }

R. Bryant introduced in his seminal paper \cite{bibdualitytheorem} a holomorphic quantity with far-reaching properties whose study has proven fertile. 

\begin{de}
Let $X \, : \, \D \rightarrow \s^3$ be a conformal immersion of representation $\Phi$ in $\R^3$ and $Z$ in $\Hy^3$ and of conformal Gauss map $Y$.
The Bryant functional of $X$ (respectively $\Phi$, $Z$) is defined  as 

$$\mathcal{Q} := \left\langle Y_{zz} , Y_{zz} \right\rangle.$$
\end{de}

In fact R. Bryant introduced the quartic $ \left\langle \partial^2 Y , \partial^2 Y\right\rangle = \mathcal{Q} dz^4$. For our purposes studying $\mathcal{Q}$ is enough.

One can draw a parallel between constant mean curvature immersions and Willmore immersions. Indeed while for a CMC immersion, the Gauss map is minimal, for a Willmore immersion the conformal Gauss map is. The Bryant functional allows us to further this comparison, as it is analogous to the Hopf functional. While the Hopf functional of a CMC immersion is holomorphic, the Bryant's functional (or Byrant's quartic) of a Willmore functional is holomorphic.

\begin{prop}
\label{XwillmoreimpliesQholomorphic}
If $X$ is Willmore then $\mathcal{Q}$ is holomorphic.
\end{prop}
\begin{proof}
If $X$ is Willmore then necessarily $Y_{z\zb} = - {\left\langle Y_z, Y_{\zb} \right\rangle } Y$, and then 

$$Y_{zz \zb} = \left( Y_{z \zb} \right)_z = - {\left(\left\langle Y_z, Y_{\zb} \right\rangle \right)_z} Y  - {\left\langle Y_z, Y_{\zb} \right\rangle }Y_z$$

and since $Y$ is conformal $$ \left\langle Y_{zz}, Y_z \right\rangle = \frac{1}{2} \left( \left\langle Y_z, Y_z \right\rangle \right)_z = 0,$$

and $$\left\langle Y_{zz}, Y \right\rangle = \left( \left\langle Y_z, Y \right\rangle \right)_z - \left\langle Y_z, Y_z \right\rangle = 0.$$

Then $$\mathcal{Q}_{\zb} =2 \left\langle Y_{zz\zb}, Y_{zz} \right\rangle=0.$$
\end{proof}

Using expression (\ref{Yzz}) in any orthonormal isotropic frame $\left( \nu, \nu^* \right)$ (that is satisfying $\left\langle \nu, \nu^*\right\rangle = -1$) of the normal bundle of $Y$ : 

$$Y_{zz} = 2 \mathcal{L}_z Y_z - \frac{\Omega_\nu}{2} \nu^* - \frac{\Omega_{\nu^*} }{2} \nu,$$

one finds 

\begin{equation}
\label{qtorduetomega}
\mathcal{Q} = -\frac{ \Omega_{\nu} \Omega_{\nu^*} }{2}.
\end{equation}

Taking $\nu$ and $\nu^*$ as in susection \ref{geometryofconformalGaussmaps} and using (\ref{omeganubase}) and (\ref{omeganustarbase}) further yields

\begin{equation}
\label{qtorduemega}
\begin{aligned}
\mathcal{Q} &= \omega^2 e^{-2 \Lambda}  \left( \left(\frac{\omega_z}{\omega} \right)_{\zb} + \frac{h^2+1 }{4} e^{2 \Lambda}  \right) \\
&= \left( \omega_{z\zb} \omega - \omega_z \omega_{\zb}  \right) e^{-2\Lambda} + \omega^2 \frac{h^2+1}{4}.
\end{aligned}
\end{equation}

The converse of proposition \ref{XwillmoreimpliesQholomorphic} is not true. 
\begin{prop}
\label{Qtorduholomorphic}
$\mathcal{Q}$ is holomorphic if and only if  there exists a holomorphic function $f$ on $\D$ such that  
\begin{equation}
\label{Qtorduholoequ}
\mathcal{W}_{\s^3} \left( X \right) = \omega \overline{f} e^{-2\Lambda}.
\end{equation}
\end{prop}

\begin{proof}

We once again use the notations of subsection \ref{formulas41} with $\nu$ and $\nu^*$ defined in (\ref{onfixenu}) and (\ref{onfixenustar}).
Then as before 
 $$\mathcal{Q}_{\zb} =2 \left\langle Y_{zz\zb}, Y_{zz} \right\rangle,$$ and using (\ref{Hnubase})  (\ref{Hnustarbase}) and  (\ref{Yzzb}) :

$$ \begin{aligned} \mathcal{Q}_{\zb} &= 2 \left\langle\left( \mathcal{W}_{\s^3} (X) \nu - \frac{|\omega|^2 e^{-2\Lambda}}{2} Y \right)_z , Y_{zz} \right\rangle \\
&=2\mathcal{W}_{\s^3} (X) \left\langle \nu_z , Y_{zz} \right\rangle   + 2 \left(\mathcal{W}_{\s^3}(X) \right)_z \left\langle \nu, Y_{zz} \right\rangle.  \end{aligned} $$

Using (\ref{YzzX}) and $\nu = \begin{pmatrix} X \\1 \end{pmatrix}$ yields 

$$  \left\langle \nu_z , Y_{zz} \right\rangle = - \frac{1}{2} \left(\omega e^{-2\Lambda} \right)_ze^{2 \Lambda}.$$

Further by (\ref{omeganubase}) $ \left\langle \nu, Y_{zz} \right\rangle = \frac{\omega}{2}$. Hence

$$\begin{aligned}
\mathcal{Q}_{\zb} &=\left(\mathcal{W}_{\s^3}(X) \right)_z \omega - \mathcal{W}_{\s^3} (X)  \left(\omega e^{-2\Lambda} \right)_ze^{2 \Lambda} \\
&= e^{2\Lambda} \omega^2 \left( \frac{ \mathcal{W}_{\s^3}(X) }{ \omega e^{-2\Lambda} } \right)_z.
\end{aligned}$$

To conclude $\mathcal{Q}$ holomorphic implies  $\left( \frac{ \mathcal{W}_{\s^3}(X) }{ \omega e^{-2\Lambda} } \right)_z = 0$, which means there exists $f$ holomorphic such that 

$$\frac{ \mathcal{W}_{\s^3}(X) }{ \omega e^{-2\Lambda} }  = \overline{f}$$
which concludes the proof.
\end{proof}

This result follow from the work of C. Bohle (see \cite{bobohle}). A. Michelat found an equivalent condition in \cite{michelatbis}.

Proposition \ref{Qtorduholomorphic} bears striking resemblance to the definition \ref{conformalWillmoreimmersion} of conformal Willmore immersions,  with the added condition that $\overline{f} \omega \in \R$. This might be better understood with the notion of isothermic immersions, which we study in the fashion of T. Rivière ((I.4) in \cite{isothermicimmersions}).

\begin{de}
\label{defisothermicimmersions}
A conformal immersion $\Phi$ of the disk $\D$ into $\R^3$ (or equivalently $X$ into $\s^3$) is said to be isothermic if around each point of $\D$ there exists a local conformal reparametrization such that $\Omega \in \R$ (equivalently $\omega \in \R$). Such a parametrization will be called isothermic, or in isothermic coordinates.
\end{de}

Isothermic immersions can be conveniently caracterized (Proposition I.1 in \cite{isothermicimmersions}).

\begin{prop}
\label{propositionisothermicimmersions}
A conformal immersion $\Phi$ of the disk $\D$ into $\R^3$ (or equivalently $X$ into $\s^3$) is isothermic if and only if there exists a non zero holomorphic function $F$ on $\D$ such that $$\Im \left( \overline{F}  \Omega \right) =0.$$
Equivalently  $X$ is isothermic if and only if there exists a non zero holomorphic function $f$ on $\D$ such that $$\Im \left( \overline{f}  \omega \right) =0.$$
In fact away from its zeros, $\sqrt{f}$ yields the conformal reparametrization into isothermic coordinates.
\end{prop}

Then (\ref{Qtorduholoequ}) not only yields that $X$ is conformal Willmore, but either $f$ is null  and then $X$ is Willmore, or  there exists a non null holomorphic $f$ such that $\overline{f} \omega \in \R$, that is $\Im \left( \overline{f}  \omega \right) =0$ i.e. $X$ is isothermic.

\begin{cor}
\label{corqtorduwillmconforme}
If $\mathcal{Q}$ is holomorphic then either $X$ is Willmore, or $X$ is conformal Willmore and isothermic.
\end{cor}

%\begin{remark}
%This result was pointed out to us by Alexis Michelat in private communications.
%\end{remark}

\section{Conformally constant mean curvature immersions}

Let $X \, : \, \D \rightarrow \s^3$ of representation $\Phi$ in $\R^3$, $Z$ in $\Hy^3$ without umbilic points and of conformal Gauss map $Y$. In this section our aim is to find a necessary and sufficient condition to have one of the three representations be conformally CMC in its immersion space. 

Let us first focus on finding a set of necessary conditions. Thanks to theorem  \ref{conformementCMC}, we know it is equivalent to  the fact that $Y$ lies in a hyperplane of $\R^{4,1}$. That is there exists constants $v \in \R^{4,1}\backslash \{ 0 \}$ and $\eta \in \R$ such that  \begin{equation} \label{conditionhyperplanveta}\left\langle Y, v \right\rangle =  \eta . \end{equation}  
Since $v$ and $\eta$ are constants, differentiating (\ref{conditionhyperplanveta}) yields 
\begin{equation} \label{conditionhyperplanvetaYz}\left\langle Y_z, v \right\rangle =  0 \end{equation} 
and
\begin{equation} \label{conditionhyperplanvetaYzb}\left\langle Y_{\zb}, v \right\rangle =  0. \end{equation} 
One can write $v$ in the moving frame $(Y, Y_z, Y_{\zb}, \nu, \nu^*)$ with $\nu$ and $\nu^*$ defined in (\ref{onfixenu}) and (\ref{onfixenustar}) :
$$v = l Y + m Y_z + n Y_{\zb} + a \nu + b \nu^*.$$
Applying (\ref{conditionhyperplanveta}), (\ref{conditionhyperplanvetaYz}) and (\ref{conditionhyperplanvetaYzb}) yields 
$$\begin{aligned}
l &= \eta \\
m &= 0 \\
n&=0 
\end{aligned}$$
And thus 
\begin{equation}
\label{ondecomposev}
v= \eta Y + a \nu + b \nu^*.
\end{equation}
$v$ can be taken such that $$\left\langle v, v \right\rangle = \kappa = \left\{ \begin{aligned} &1 \text{ if } v \text{ is spacelike} \\ &0 \text{ if } v \text{ is lightlike }  \\ &-1 \text{ if } v \text{ is timelike}. \end{aligned} \right.$$
% and $v_5= 1$ if $v$ is lightlike.
From this decomposition we will deduce characterizations of $a$ and $b$.
Since $v$ is constant one can differentiate (\ref{ondecomposev}) and put formulas (\ref{nuzformule}) and (\ref{nustarzformule}) to effect :
$$\begin{aligned}
0& = \left( \eta - aH_{\nu} -b H_{\nu^*} \right)Y_z   +\left( a_z - a \left\langle \nu_z,\nu^* \right\rangle \right) \nu  + \left( b_z  - b \left\langle \nu^*_z , \nu \right\rangle \right) \nu^*-\frac{\left( a \Omega_\nu +b \Omega_{\nu^*} \right)}{ |\omega|^2 e^{-2\Lambda} } Y_{\zb} \\
&= \left( \eta + \frac{2b \mathcal{W}_{\s^3} \left(X \right)}{|\omega|^2 e^{-2\Lambda} } \right) Y_z + \left( a_z - a \left\langle \nu_z,\nu^* \right\rangle \right) \nu  + \left( b_z  - b \left\langle \nu^*_z , \nu \right\rangle \right) \nu^*-\frac{\left( a \Omega_\nu +b \Omega_{\nu^*} \right)}{ |\omega|^2 e^{-2\Lambda} } Y_{\zb} 
\end{aligned}$$
with (\ref{Hnubase}) and (\ref{Hnustarbase}). Further since $\left\langle \nu^*_z , \nu \right\rangle = \left( \langle \nu, \nu^* \rangle \right)_z - \langle \nu_z, \nu^* \rangle = \frac{\overline{\omega}_z }{ \overline{\omega}} $, using (\ref{nuznustarbase}), we find 
$$ 0 = \left( \eta + \frac{2b \mathcal{W}_{\s^3} \left(X \right)}{|\omega|^2 e^{-2\Lambda} } \right) Y_z + \left( a_z + a\frac{\overline{\omega}_z }{ \overline{\omega}}  \right) \nu  + \left( b_z  - b\frac{\overline{\omega}_z }{ \overline{\omega}} \right) \nu^*-\frac{\left( a \Omega_\nu +b \Omega_{\nu^*} \right)}{ |\omega|^2 e^{-2\Lambda} } Y_{\zb}.$$
Besides $$\langle v, v \rangle = \eta^2 -2ab,$$
and since $Y$, $\nu$ and $\nu^*$ are bounded in $\R^5$ away from umbilic points, $a,b < \infty$.
Then $a,b$ are real functions and $\eta$ a real constant such that 
\begin{equation}
\label{systemeCMC1}
a_z + a \frac{\overline{\omega}_z}{\overline{\omega}}=0,
\end{equation}
\begin{equation}
\label{systemeCMC2}
b_z-b \frac{\overline{\omega}_z}{\overline{\omega}}=0, 
\end{equation}
\begin{equation}
\label{systemeCMC3}
2b \mathcal{W}_{\s^3} (X) + \eta |\omega|^2 e^{-2\Lambda} = 0,
\end{equation}
\begin{equation}
\label{systemeCMC4}
a \Omega_\nu + b \Omega_{\nu^*} = 0 ,
\end{equation}
\begin{equation}
\label{systemeCMC5}
  ab = -\frac{\left\langle v,v \right\rangle - \eta^2 }{2} \text{ real constant.}
\end{equation}
One can recast (\ref{systemeCMC1}) as $a_z \overline{\omega} +a \overline{\omega}_z = 0$, or rather since $a \in \R$
$$a_{\zb} \omega + a \omega_{\zb} = 0.$$
This yields $$ \left( a \omega \right)_{\zb} = 0,$$
i.e. there exists $f \,: \, \D \rightarrow \C$ holomorphic (since $a \omega < \infty$)  such that  \begin{equation} \label{aomegaf} a \omega = f.\end{equation}
One then has   $\overline{f} \omega = \overline{a \omega} \omega  = a |\omega|^2 \in \R$ since $a \in \R$. Then according to proposition \ref{propositionisothermicimmersions}, unless $f = 0$ on $\D$, $X$ is isothermic.
Working similarly on (\ref{systemeCMC2}) one finds there exists $g$ holomorphic (since $b < \infty$ and $\omega \neq 0$ by hypothesis) on $\D$ such that \begin{equation} \label{bomegaf} b = g \omega. \end{equation} 
Then, if $g$ is not null on $\D$, working away from its zeros yields $$\overline{ \frac{1}{g} } \omega = \frac{|\omega|^2}{\overline{b} } \in \R$$ since $b \in \R$.  Then according to proposition \ref{propositionisothermicimmersions}  $X$ is isothermic.
So unless $f= g= 0$ on $\D$, $X$ is isothermic. If $f=g=0$, then (\ref{systemeCMC3}) ensures $\eta=0$ which in turn yields $v=0$, a case excluded from the start of this reasoning. As a consequence we get our first necessary condition :

\vspace{5mm}
\underline{\textbf{ $X$ is  isothermic.}}
\vspace{5mm}

To go further one can reframe (\ref{systemeCMC3}) in terms of $f$ and $g$. 
Indeed 
$$2b \mathcal{W}_{\s^3} (X) + \eta |\omega|^2 e^{-2\Lambda} = \omega \left(2g \mathcal{W}_{\s^3} (X) + \eta  \overline{\omega} e^{-2\Lambda} \right)$$ with $\omega \neq 0$ ensuring that (\ref{systemeCMC3}) is equivalent to 
$$2 g\mathcal{W}_{\s^3} (X) + \eta  \overline{\omega} e^{-2\Lambda}=0.$$ 
This implies that if $g(z_0)= 0$ for any given $z_0$ in $\D$, then $\eta =0$, and with (\ref{systemeCMC4}) $f(z_0)=0$. So $v(z_0)= 0$ and since $v$ is a constant $v=0$, which is a contradiction. Then $g$ has no zero on $\D$.
Letting $\varphi = \frac{1}{g}$ be a holomorphic function on $\D$, one finds (\ref{systemeCMC3}) to be equivalent to
\begin{equation}
\label{systemeCMC3bis}
 \mathcal{W}_{\s^3} (X) =-\frac{\eta}{2} \varphi \overline{\omega} e^{-2 \Lambda} =\overline{\left( -\frac{\eta}{2} \varphi \right)} {\omega} e^{-2 \Lambda}.
\end{equation}
Consequently, proposition \ref{Qtorduholomorphic} implies our second necessary condition

\vspace{5mm}
\underline{\textbf{$\mathcal{Q}$ is holomorphic.}}
\vspace{5mm}

Similarly $$ \begin{aligned} a \Omega_\nu + b \Omega_{\nu^*} &= a \omega  \frac{\Omega_\nu}{\omega} + \frac{b}{\omega} \omega \Omega_{\nu^*} \\
&=  a\omega + \frac{b}{\omega} \Omega_\nu \Omega_{\nu^*} \text{ using (\ref{omeganubase})} \\
&=a\omega - 2\frac{b}{\omega}  \mathcal{Q}  \text{ using (\ref{qtorduetomega})}.
\end{aligned}$$
This yields that  (\ref{systemeCMC4}) is equivalent to 
\begin{equation}
\label{qtorduforme}
\begin{aligned}
\mathcal{Q} &= \frac{a \omega^2}{2b} =  \frac{f}{2g}= \frac{fg}{2g^2} = \frac{1}{2}ab \varphi^2 = \frac{\eta^2 - \kappa }{4} \varphi^2. 
\end{aligned}
\end{equation}

Summing up our analysis has given us two necessary conditions :

\begin{itemize}
\item \underline{\textbf{$X$ is isothermic, with $\varphi \overline{\omega} \in \R$}}
\item \underline{\textbf{$\mathcal{Q}$ is holomorphic, with  $\mathcal{Q} = \frac{ \eta^2 - \kappa}{4} \varphi^2$.}}
\end{itemize}

Let us show they are necessary.
\vspace{5mm}

Let $X$ be an isothermic immersion such that $\mathcal{Q}$ is holomorphic.  Our aim is to write $\mathcal{Q}$ and $\mathcal{W}_{\s^3}(X)$ in the forms  respectively of (\ref{qtorduforme}) and (\ref{systemeCMC3bis}).

  Since $X$ is isothermic there exists a non null holomorphic function $\varphi_0$ such that $$R :=\overline{\varphi_0} \omega  \in \R.$$

\vspace{5mm}
\noindent \underline{\textbf{ Claim 1 : there exists a constant  $m \in \R$ such that $\mathcal{Q} = m \varphi_0^2$.}}
\vspace{5mm}
\begin{proof}
We will write $\mathcal{Q}$ as a function of $\varphi$, using (\ref{qtorduemega}) : 
$$
\mathcal{Q} =  \left(\omega_{z\zb} \omega - \omega_z \omega_{\zb} \right)e^{-2\Lambda} + \omega^2 \frac{h^2+1}{4}$$
Since $ \omega = \frac{R}{\overline{\varphi_0}}$, 
$$
\begin{aligned}
\omega_z &= \frac{R_z}{\overline{\varphi_0}} \\
\omega_{\zb} &= \frac{R_{\zb} }{\overline{\varphi_0}} - \frac{\overline{\partial_z \varphi_0} R}{\overline{\varphi_0^2}} \\
\omega_{z \zb} &= \frac{R_{z \zb} }{\overline{\varphi_0}} -\frac{\overline{\partial_z \varphi_0}R_z}{\overline{\varphi_0^2}}.\end{aligned}
$$
Thus 
\begin{equation}
\label{leomegazzb}
\begin{aligned}
\omega_{z\zb} \omega - \omega_z \omega_{\zb} &= \frac{ R}{\overline{\varphi_0}} \left( \frac{R_{z \zb} }{\overline{\varphi_0}} -\frac{\overline{\partial_z \varphi_0}R_z}{\overline{\varphi_0^2}} \right) -  \frac{R_z}{\overline{\varphi_0}} \left(  \frac{R_{\zb} }{\overline{\varphi_0}} - \frac{\overline{\partial_z \varphi_0} R}{\overline{\varphi_0^2}} \right) \\
&= \frac{ R_{z \zb} R - R_z R_{\zb} }{\overline{\varphi_0^2} } \\
&=\left( \frac{ R_{z \zb} R - R_z R_{\zb} }{\left| \varphi_0 \right|^4 } \right) \varphi_0^2. 
\end{aligned}
\end{equation}

As announced $\mathcal{Q}$ can be expressed  : 
\begin{equation}
\label{qtordumieux}
\mathcal{Q} = \frac{ \left(R_{z \zb} R - R_z R_{\zb} \right) e^{-2 \Lambda} + \frac{h^2 +1}{4}R^2}{\left| \varphi_0 \right|^4 } \varphi_0^2.
\end{equation}
Since $R \in \R$,  $ \frac{ \left(R_{z \zb} R - R_z R_{\zb} \right) e^{-2 \Lambda} + \frac{h^2 +1}{4}R^2}{\left| \varphi_0 \right|^4 }$ is real. Further $$\left(  \frac{ \left(R_{z \zb} R - R_z R_{\zb} \right) e^{-2 \Lambda} + \frac{h^2 +1}{4}R^2}{\left| \varphi_0 \right|^4 } \right)_{\zb} = \left( \frac{\mathcal{Q} }{\varphi_0^2} \right)_{\zb} =0$$
since $\mathcal{Q}$ and $\varphi_0$ are holomorphic.
As a real holomorphic function $ \frac{ \left(R_{z \zb} R - R_z R_{\zb} \right) e^{-2 \Lambda} + \frac{h^2 +1}{4}R^2}{\left| \varphi_0 \right|^4 }$ is  necessarily a constant that we will denote $m$.  This proves claim 1.
\end{proof}

\vspace{5mm}
\underline{\textbf{Claim 2 : There exists $n \in \R$ such that $\mathcal{W}_{\s^3} (X) =n \omega \overline{\varphi_0} e^{-2\Lambda}.$}}
\vspace{5mm}
\begin{proof}
Proposition \ref{Qtorduholomorphic} yields  $f$ holomorphic on $\D$ such that $$\mathcal{W}_{\s^3} (X) = \omega \overline{f} e^{-2\Lambda}.$$
Using $\omega= \frac{R}{\overline{\varphi_0}}$ one deduces
$$ \overline{ \left( \frac{f}{\varphi_0} \right)} = \frac{ \mathcal{W}_{\s^3} (X)}{R e^{-2\Lambda}} \in \R.$$
Since $ \frac{f}{\varphi_0}$ is holomorphic, there exists $n \in \R$ such that $f = n \varphi_0$, which proves claim 2.
\end{proof}

\vspace{5mm}
\underline{\textbf{Claim 3 : There exists $\lambda \in \R$, $\kappa \in \{ -1,0,1 \}$ and $\eta \in \R$  such that }} \newline
\underline{\textbf{ $\mathcal{W}_{\s^3} (X) =-\lambda \frac{\eta}{2} \omega \overline{\varphi_0} e^{-2\Lambda}$ and  $\mathcal{Q} = \frac{ \eta^2 - \kappa}{4} \lambda^2 \varphi_0^2$.}}
\vspace{5mm}
\begin{proof}
 If $n^2 -m \neq 0$, let $\lambda =2 \sqrt{ \left| n^2-m \right|}$, $\kappa = sg( n^2-m)$ and $\eta = -2 \frac{n}{\lambda}$.
Then $n = -\lambda \frac{\eta}{2}$ and 
$$ \begin{aligned} m &= -\left( n^2-m \right) + n^2\\
&=- \frac{\lambda^2 \kappa}{4} +  \lambda^2 \frac{\eta^2}{4}\\
&= \lambda^2 \frac{  \eta^2 - \kappa}{4}.
\end{aligned}$$
If $n^2 = m $, let $\kappa = 0$, $\lambda=1$, $\eta = -2 n$, which concludes the proof of claim 3.
\end{proof}

In the following we set $\varphi = \lambda \varphi_0$.

\vspace{5mm}
\underline{\textbf{Claim 4 : $v = \eta Y + \frac{\eta^2- \kappa}{2} \frac{\varphi}{\omega} \varphi \nu + \frac{\omega}{\varphi} \nu^*$ is a constant vector in $\R^{4,1}$.}}
\vspace{5mm}
\begin{proof}
Since $\eta \in \R$, $$a := \frac{\eta^2- \kappa}{2} \frac{\varphi}{\omega} = \frac{\eta^2- \kappa}{2}\frac{|\varphi|^2}{ \omega \overline{\varphi}} \in \R$$ and $$b := \frac{\omega}{\varphi} = \frac{ \omega \overline{\varphi} }{ |\varphi|^2} \in \R,$$
$v$ does belong in $\R^{4,1}$.
Further $$a_{\zb} + a \frac{ \omega_{\zb} }{\omega} =  \frac{\eta^2- \kappa}{2} \varphi \left( \frac{- \omega_{\zb} }{\omega^2} + \frac{ \omega_{\zb} }{ \omega^2}  \right)=0$$ 
and
$$b_{\zb} -b \frac{ \omega_{\zb} }{\omega} =  \frac{1}{ \varphi} \left( \omega_{\zb}-   \omega_{\zb}  \right)=0,$$
meaning that $a$ and $b$ satisfy (\ref{systemeCMC1}) and (\ref{systemeCMC2}).
Besides 
$$2 b \mathcal{W}_{\s^3} ( X)  + \eta |\omega|^2 e^{-2 \Lambda} = -2 \frac{ \eta}{2} \overline{ \omega} \varphi e^{-2 \Lambda}  \frac{\omega}{ \varphi} + \eta |\omega|^2 e^{-2 \Lambda} = 0$$
since by design, see claim 3, $\mathcal{W}_{\s^3} (X) =- \frac{\eta}{2} \omega \overline{\varphi} e^{-2\Lambda} = - \frac{\eta}{2} \overline{\omega} \varphi e^{-2 \Lambda}.$ $v$ must then satisfy   (\ref{systemeCMC3}).
Once more by construction $\mathcal{Q}$ satisfies (\ref{qtorduforme}), which was shown to be equivalent to (\ref{systemeCMC4}).
$v$ then satisfies : $v_z = 0$, and $v$ is a constant in $\R^{4,1}$, which proves claim 4. \end{proof}
$Y$ is then hyperplanar and according to theorem \ref{conformementCMC} $X$ is conformally CMC in  a space depending entirely on $\langle v, v \rangle = \kappa$. 
$\kappa$ can be expressed explicitely from $\mathcal{Q}$ et $\mathcal{W}_{\s^3} ( X)$. Indeed 
$$\begin{aligned} \left( \mathcal{W}_{\s^3} ( X)\right)^2 - \overline{\omega}^2e^{-4\Lambda} \mathcal{Q} &= \frac{ \eta^2}{4} \overline{\omega}^2  \varphi^2 e^{-4 \Lambda} - \frac{ \eta^2 - \kappa}{4} \varphi^2 \overline{ \omega}^2 e^{-4 \Lambda} \text{ using Claim 3 }\\
&= \kappa \left( \frac{\varphi \overline{\omega} e^{-2 \Lambda} }{2} \right)^2.
\end{aligned}$$

Since $\varphi \overline{\omega} \in \R^*$, $\left( \frac{\varphi \overline{\omega} e^{-2 \Lambda} }{2} \right)^2 \in \R_+^*$ and necessarily :

\begin{equation}
\label{lekappaaaaaaaaaaaaaaaaa}
\kappa = sg \left( \left( \mathcal{W}_{\s^3} ( X)\right)^2 - \overline{\omega}^2e^{-4\Lambda} \mathcal{Q} \right).
\end{equation}

We deduce the following theorem.
\begin{theo}
\label{conformementCMCnya1}
Let $X$ be a smooth conformal immersion on $\D$ in $\s^3$, and $\Phi$ (respectively $Z$) its representation in $\R^3$ (respectively $\mathbb{H}^3$) through $\pi$ (respectively $\tilde \pi$). 
 We assume that $X$ (or equivalently, see (\ref{projectionstereo5}) and (\ref{projectionhyper5}), $\Phi$ or $Z$)  has no umbilic point.
One of the representation of $X$ is conformally $CMC$ in its ambiant space if and only if $\mathcal{Q}$ is holomorphic and $X$ is isothermic.
More precisely  $\left( \mathcal{W}_{\s^3} ( X)\right)^2 - \overline{\omega}^2e^{-4\Lambda} \mathcal{Q}$ is then necessarily real and
\begin{itemize}
\item
$\Phi$ is conformally CMC (respectively minimal) in $\R^3$ if and only if $$\left( \mathcal{W}_{\s^3} ( X)\right)^2 - \overline{\omega}^2e^{-4\Lambda} \mathcal{Q}=0.$$ 
\item 
$X$  is conformally CMC (respectively minimal) in $\s^3$ if and only if $$\left( \mathcal{W}_{\s^3} ( X)\right)^2 - \overline{\omega}^2e^{-4\Lambda} \mathcal{Q} < 0.$$
\item
$Z$ is conformally CMC (respectively minimal) in $\Hy^3$ if and only if $$\left( \mathcal{W}_{\s^3} ( X)\right)^2 - \overline{\omega}^2e^{-4\Lambda} \mathcal{Q} > 0.$$
\end{itemize}
Conformally minimal immersions satisfy $\mathcal{W}_{\s^3} (X)=0$.
\end{theo}

Notice especially that according to our analysis $X$ isothermic and $\mathcal{Q}$ holomorphic heavily determines $\mathcal{Q}$. As a matter of fact it ensures that $\overline{\omega^2} \mathcal{Q} \in \R$. Accordingly one can slightly change the hypotheses of theorem \ref{conformementCMCnya1}.

\begin{theo}
\label{conformementCMCnya2}
Let $X$ be a smooth conformal immersion on $\D$ in $\s^3$, and $\Phi$ (respectively $Z$) its representation in $\R^3$ (respectively $\mathbb{H}^3$) through $\pi$ (respectively $\tilde \pi$). 
 We assume $X$ (or equivalently, see (\ref{projectionstereo5}) and (\ref{projectionhyper5}), $\Phi$ or $Z$) has no umbilic point.
One of the representation of $X$ is conformally $CMC$ in its ambiant space if and only if $\mathcal{Q}$ is holomorphic and $\overline{\omega^2} \mathcal{Q} \in \R$.
More precisely  
\begin{itemize}
\item
$\Phi$ is conformally CMC (respectively minimal) in $\R^3$ if and only if $$\left( \mathcal{W}_{\s^3} ( X)\right)^2 - \overline{\omega}^2e^{-4\Lambda} \mathcal{Q}=0.$$
\item 
$X$  is conformally CMC (respectively minimal) in $\s^3$ if and only if $$\left( \mathcal{W}_{\s^3} ( X)\right)^2 - \overline{\omega}^2e^{-4\Lambda} \mathcal{Q} < 0.$$
\item
$Z$ is conformally CMC (respectively minimal) in $\Hy^3$ if and only if $$\left( \mathcal{W}_{\s^3} ( X)\right)^2 - \overline{\omega}^2e^{-4\Lambda} \mathcal{Q} > 0.$$
\end{itemize}
Conformally minimal immersions satisfy $\mathcal{W}_{\s^3} (X)=0$.
\end{theo}

\begin{proof}
If $X$ is conformally CMC, then $\mathcal{Q}$ is holomorphic and  $\left( \mathcal{W}_{\s^3} ( X)\right)^2 - \overline{\omega}^2e^{-4\Lambda} \mathcal{Q}$ is real according to theorem \ref{conformementCMCnya1}. Then since $\mathcal{W}_{\s^3} ( X) \in \R$, 
$\overline{\omega^2} \mathcal{Q} \in \R$.

Conversely assume that  $\mathcal{Q}$ is holomorphic and $\overline{\omega^2} \mathcal{Q} \in \R$. Then using corollary \ref{corqtorduwillmconforme}, $X$ is isothermic and conformal Willmore or Willmore. If $X$ is isothermic, the theorem is proved with theorem \ref{conformementCMCnya1}. Let us then assume that $X$ is Willmore. 
Let us first assume that $\mathcal{Q}$ is non null.
Away from the zeros of $\mathcal{Q}$, $\overline{\omega^2} \mathcal{Q}$  does not cancel and is then of fixed sign, and $\sqrt{ \mathcal{Q}}$ is holomorphic.
Then $$  \left(\overline{\omega} \sqrt{\mathcal{Q}} \right)^2 \in \R^*,$$ and thus $$\overline{\omega} \sqrt{\mathcal{Q}} \in \R \text{ or } i \R.$$
There exists then a non null holomorphic function ($\varphi =\sqrt{\mathcal{Q} }$ or $\varphi =i \sqrt{\mathcal{Q}}$) such that $\overline{\omega} \varphi \in \R$.  The theorem is then proved with theorem \ref{conformementCMCnya1}. 
The case $X$ Willmore and $\mathcal{Q} =0$  is now the only one left. Using theorem C in \cite{bibdualitytheorem}  yields $\Phi$ conformally minimal in $\R^3$. This concludes the proof.
\end{proof}

\renewcommand{\thesection}{\Alph{section}}
\setcounter{section}{0} 
\section{Appendix}
\tocless\subsection{Formulas in $\R^3$}
Let $\Phi \, : \, \D \rightarrow \R^3$ be a smooth conformal immersion.  Let $\n = \frac{\Phi_z \times \Phi_{\zb} }{ i \left| \Phi_z \right|^2 }$ denote its Gauss map (with $\times$ the classical vectorial product in $\R^3$), $\lambda = \frac{1}{2}  \log \left( 2 \left| \Phi_z \right|^2 \right)$ its conformal factor and $H = \left\langle \frac{ \Phi_{z \zb}}{ \left| \Phi_z \right|^2} , \n \right\rangle $ its mean curvature. Its tracefree curvature is defined as follows 
$$\Omega:= 2\left\langle \Phi_{zz}, \n \right\rangle. $$

Then 
\begin{equation}
\label{nzphi}
\n_z = - H \Phi_z - \Omega e^{-2 \lambda} \Phi_{\zb},
\end{equation} 
\begin{equation}
\label{phizzb}
\Phi_{z \zb} = H \frac{e^{2\lambda}}{2} \n ,
\end{equation}
\begin{equation}
\label{phizz}
\Phi_{z z} = 2 \lambda_z \Phi_z + \frac{\Omega}{2} \n
\end{equation}
and Gauss-Codazzi can be written 
\begin{equation}
\label{Gauss-CodazziR3}
\Omega_{\zb} e^{-2\lambda} = H_z .
\end{equation}

Further if we write the second fundamental form of $\Phi$, $ A = \langle \nabla^2 \Phi , \n \rangle $ such that 
$$ e^{-2\lambda} A= \begin{pmatrix} \epsilon & \varphi \\ \varphi & \gamma \end{pmatrix},$$ then
$$ H= \frac{ \epsilon + \gamma }{2},$$
$$ \Ar =  \begin{pmatrix} \frac{\epsilon - \gamma }{2} & \varphi \\ \varphi & \frac{\gamma - \epsilon}{2} \end{pmatrix},$$ with $\Ar$ the tracefree second fundamental form defined in (\ref{tracefreesecondfundamentalform}) and
\begin{equation} \label{nablanhar} \begin{aligned} \nabla \n &= - H \nabla \Phi - \Ar \nabla \Phi = \begin{pmatrix} \epsilon \Phi_x + \varphi \Phi_y \\ \varphi \Phi_x + \gamma \Phi_y \end{pmatrix} \\
&= - H \nabla \Phi - \begin{pmatrix} \frac{\epsilon - \gamma }{2} \Phi_x + \varphi \Phi_y \\ \varphi \Phi_x -  \frac{\epsilon - \gamma }{2} \Phi_y \end{pmatrix}.  \end{aligned} \end{equation}
 We can check 
$$
\begin{aligned} 
\n \times \Ar \nabla \Phi &= \n \times \begin{pmatrix} \frac{\epsilon - \gamma }{2}  \Phi_x + \varphi \Phi_y \\ \varphi \Phi_x - \frac{\epsilon - \gamma }{2}  \Phi_y \end{pmatrix}  \\
&=  \begin{pmatrix} \frac{\epsilon - \gamma }{2} \Phi_y - \varphi \Phi_x \\ \varphi \Phi_y + \frac{\epsilon - \gamma }{2} \Phi_x \end{pmatrix}
\end{aligned}$$
and notice 
\begin{equation}
\label{formuleennablaperpn} \nabla^\perp \n = - H \nabla^\perp \Phi - \n \times \Ar \nabla \Phi.
\end{equation}

\tocless\subsection{Formulas in $\s^3$}
\label{sectionannexcalculss3}
Let $\Phi \, : \, \D \rightarrow \R^3$ be a smooth conformal immersion and $X = \pi^{-1} \circ \Phi \, : \, \D \rightarrow \s^3$. Let $\Lambda := \frac{1}{2} \log \left( 2 \left| X_z \right|^2 \right)$ be its conformal factor, $\vec{N}$  such that $\left( X, e^{-\Lambda} X_x, e^{- \Lambda} X_y, \vec{N} \right)$ is a direct orthonormal basis of $\R^4$ its Gauss map,  $ h = \left\langle  \frac{X_{z\zb}}{\left| X_z \right|^2}, \vec{N} \right\rangle$ its mean curvature and  $\omega := 2\left\langle X_{zz}, \n \right\rangle$ its tracefree curvature. 
Then
\begin{equation}
\label{projectionstereo0}
X :=  \frac{1}{1+ |\Phi|^2} \begin{pmatrix} 2\Phi \\ |\Phi|^2-1 \end{pmatrix}
\end{equation}
which yields
\begin{equation}
\label{projectionstereo1}
 \begin{aligned} 
X_z &= d\pi^{-1} \left( \Phi_z \right)= \frac{2}{1+|\Phi|^2 } \begin{pmatrix} \Phi_z \\ 0 \end{pmatrix} -\frac{ 4 \langle \Phi_z, \Phi \rangle_3}{ \left( 1+ |\Phi|^2 \right)^2 } \begin{pmatrix} \Phi \\ -1\end{pmatrix}.
\end{aligned}
\end{equation}

Since $\pi$ is conformal, $\left\langle d \pi^{-1} \left( \Phi_z \right), d \pi^{-1} \left( \n \right) \right\rangle =  \left\langle \Phi_z, \n \right\rangle=0$. Then $ \vec{N} = \frac{ d \pi^{-1} \left( \n \right)}{ \left| d \pi^{-1} \left( \n \right) \right|}$ and thus
\begin{equation}
\label{projectionstereo2}
\begin{aligned}
\vec{N} &=\begin{pmatrix} \n \\ 0 \end{pmatrix} - \frac{2 \langle \n, \Phi \rangle }{1+ |\Phi|^2 } \begin{pmatrix} \Phi \\ -1 \end{pmatrix}.
\end{aligned}
\end{equation}
Using the corresponding definitions we successively deduce
\begin{equation}
\label{projectionstereo3}
\begin{aligned}
e^{2 \Lambda} &=2 \left\langle X_z, X_{\zb} \right\rangle = \frac{4}{\left( 1+ |\Phi|^2 \right)^2 } e^{2\lambda} ,
\end{aligned}
\end{equation}
\begin{equation}
\label{projectionstereo4}
\begin{aligned}
h &= \left\langle \frac{X_{z \zb}}{ \left| X_z \right|^2}, \vec{N} \right\rangle = \frac{|\Phi|^2 +1}{2}H + \langle \n , \Phi \rangle_3 
\end{aligned}
\end{equation}
\begin{equation}
\label{projectionstereo5}
\begin{aligned}
\omega &= 2 \left\langle X_{zz}, \vec{N} \right\rangle= \frac{2\Omega}{1+ | \Phi|^2 }.
\end{aligned}
\end{equation}
Then one can compute 
\begin{equation}
\label{projectionstereo7}
 \begin{aligned}
 h \begin{pmatrix} X \\1 \end{pmatrix} +  \begin{pmatrix} \vec{N} \\ 0 \end{pmatrix} &=  \left( \frac{|\Phi|^2 +1}{2}H + \langle \n , \Phi \rangle  \right) \begin{pmatrix}  \frac{2 \Phi}{1+ |\Phi|^2} \\ \frac{|\Phi|^2-1}{1+ |\Phi|^2}  \\ 1 \end{pmatrix}+  \begin{pmatrix} \n -  \frac{2 \langle \n, \Phi \rangle }{1+ |\Phi|^2 } \Phi \\  \frac{2 \langle \n, \Phi \rangle }{1+ |\Phi|^2 } \\ 0  \end{pmatrix} \\
&= \begin{pmatrix} H \Phi + \frac{2 \langle \n, \Phi \rangle }{1+ |\Phi|^2 } \Phi \\ H \frac{ | \Phi |^2 - 1 }{2} + \langle \n , \Phi \rangle \frac{|\Phi|^2-1}{1+ |\Phi|^2} \\ H \frac{ | \Phi|^2 +1}{2} + \langle \n, \Phi \rangle \end{pmatrix}  +  \begin{pmatrix} \n -  \frac{2 \langle \n, \Phi \rangle }{1+ |\Phi|^2 } \Phi \\  \frac{2 \langle \n, \Phi \rangle }{1+ |\Phi|^2 } \\ 0  \end{pmatrix} \\
&= H \begin{pmatrix} \Phi \\ \frac{ | \Phi|^2 -1}{2} \\ \frac{ | \Phi |^2 +1}{2} \end{pmatrix} + \begin{pmatrix} \n \\ \langle \n ,\Phi \rangle \\ \langle \n , \Phi \rangle \end{pmatrix}.
\end{aligned}
\end{equation}
Which shows that 
\begin{equation}
\label{projectionstereo6}
Y = h \begin{pmatrix} X \\ 1 \end{pmatrix} + \begin{pmatrix} N \\ 0 \end{pmatrix}.
\end{equation}

One may wish to compute in $\s^3$ without going through $\Phi$. The relevant formulas then are 
\begin{equation}
\label{NzX}
\vec{N}_z = - h X_z - \omega e^{-2 \Lambda} X_{\zb},
\end{equation} 

\begin{equation}
\label{Xzzb}
X_{z \zb} = h \frac{e^{2\Lambda}}{2} \vec{N} - \frac{e^{2\Lambda}}{2} X ,
\end{equation}

\begin{equation}
\label{Xzz}
X_{z z} = 2 \Lambda_z X_z + \frac{\omega}{2} \vec{N},
\end{equation}

and Gauss-Codazzi can be written 

\begin{equation}
\label{Gauss-Codazzis3}
\omega_{\zb} e^{-2\Lambda} = h_z .
\end{equation}

\tocless\subsection{Mean curvature of a sphere in $\s^3$}
\label{meancurvsphere}
Let $\sigma$ be a sphere in $\s^3$. Up to an isometry of $\s^3$ $\sigma$ can be assumed to be a sphere centered on the south pole $S$ of radius $r \le \frac{\pi}{2}$. Then $\pi \circ \sigma$ is a sphere  of $\R^3$ centered on the origin of radius $R \le 1$. It can be conformally parametrized over $\R^2 \cup \infty$ by  $\Phi (x,y) =  \frac{R}{1+ x^2 +y^2} \begin{pmatrix} 2x \\2 y \\ x^2 +y^2-1 \end{pmatrix}$, of constant mean curvature $H = \frac{1}{R}$.  Then $\sigma$ is conformally parametrized by 

$$X =  \frac{ 1}{1+ R^2} \begin{pmatrix}  \frac{2R}{1+ x^2 +y^2} \begin{pmatrix} 2x \\2 y \\ x^2 +y^2-1 \end{pmatrix} \\ R^2 -1  \end{pmatrix}.$$

One can easily compute using basic trigonometry the tangent of $r$ (see drawing to insert) and find 

$$ \begin{aligned}
\tan \left( r \right) &= \frac{2R }{1-R^2}. 
\end{aligned}$$
Computing $h$ at any point $(x,y)$ using (\ref{projectionstereo4}) yields  with  $H= \frac{1}{R}$, $\n = - \frac{\Phi}{R}$

$$h= \frac{R^2 +1}{2R} - R = \frac{1}{\tan (r)}$$
for any $(x,y)$.

Since neither $h$ nor $r$ change under the action of isometries, any sphere $\sigma$ of $\s^3$ of radius $r$ has constant mean curvature 

\begin{equation}
\label{h(r)}
h = \text{cotan} (r). 
\end{equation}

\tocless\subsection{Formulas in $\mathbb{H}^3$}
\label{sectionannexcalculsH3}
Let $\Phi \, : \, \D \rightarrow \R^3$ be a smooth conformal immersion and $Z = {\tilde \pi}^{-1} \circ \Phi \, : \, \D \rightarrow \mathbb{H}^3$. 
Then
\begin{equation}
\label{projectionhyper0}
Z :=  \frac{1}{1- |\Phi|^2} \begin{pmatrix} 2\Phi \\ |\Phi|^2+1 \end{pmatrix}
\end{equation}
which yields
\begin{equation}
\label{projectionhyper1}
\begin{aligned}
Z_z &= \frac{2}{1-  | \Phi|^2 } \begin{pmatrix} \Phi_z \\ 0 \end{pmatrix} + \frac{4 \langle \Phi_z , \Phi \rangle }{\left( 1- | \Phi|^2 \right)^2} \begin{pmatrix} \Phi \\ 1 \end{pmatrix} .
\end{aligned}
\end{equation}

Since $\tilde \pi$ is conformal, $\left\langle d \tilde \pi^{-1} \left( \n \right) ,Z_z\right\rangle = \left \langle \Phi_z , \n \right\rangle = 0$. Then $ \n^Z = \frac{d \tilde \pi^{-1} (\n)}{ \left| d \tilde \pi^{-1} \left( \n \right) \right|}$ and thus 
\begin{equation}
\label{projectionhyper2}
\begin{aligned}
\n^Z &= \begin{pmatrix} \n \\ 0 \end{pmatrix} + \frac{2 \langle \n, \Phi \rangle }{1- |\Phi|^2} \begin{pmatrix} \Phi \\ 1 \end{pmatrix}.
\end{aligned}
\end{equation}
Using the corresponding definition we successively deduce
\begin{equation}
\label{projectionhyper3}
\begin{aligned}
e^{2\lambda^Z} &= \frac{4}{\left( 1- |\Phi|^2 \right)^2 } e^{2\lambda} ,
\end{aligned}
\end{equation}
\begin{equation}
\label{projectionhyper4}
\begin{aligned}
H^{Z} &=  \frac{1- |\Phi|^2}{2} H - \langle \n , \Phi \rangle ,
\end{aligned}
\end{equation}
\begin{equation}
\label{projectionhyper5}
\begin{aligned}
\Omega^{Z} &= \frac{2\Omega}{1-|\Phi|^2}.
\end{aligned}
\end{equation}
Then one can compute 
\begin{equation}
\label{projectionhyper7}
 \begin{aligned}
H^Z \begin{pmatrix} Z_h \\ -1 \\ Z_4 \end{pmatrix} + \begin{pmatrix} \n^Z_h \\ 0 \\ \n^Z_4 \end{pmatrix} &=  \left( \frac{1- |\Phi|^2}{2} H - \langle \n , \Phi \rangle   \right) \begin{pmatrix}  \frac{2 \Phi}{1- |\Phi|^2}  \\ -1 \\   \frac{|\Phi|^2+1}{1-|\Phi|^2}  \end{pmatrix} +  \begin{pmatrix} \n + \frac{2 \langle \n, \Phi \rangle }{1- |\Phi|^2 } \Phi  \\ 0 \\  \frac{2 \langle \n,\Phi \rangle }{1- |\Phi|^2 }  \end{pmatrix} \\
&= \begin{pmatrix} H \Phi - \frac{2 \langle \n, \Phi \rangle }{1- |\Phi|^2 } \Phi \\- \frac{1- |\Phi|^2}{2} H + \langle \n , \Phi \rangle  \\ H \frac{ | \Phi |^2 +1 }{2} - \langle \n, \Phi \rangle  \frac{|\Phi|^2+1}{1-|\Phi|^2} \end{pmatrix} +  \begin{pmatrix} \n +  \frac{2 \langle \n, \Phi \rangle }{1- |\Phi|^2 } \Phi  \\ 0 \\ \frac{2 \langle \n, \Phi \rangle }{1- |\Phi|^2 }   \end{pmatrix} \\
&= H \begin{pmatrix} \Phi \\ \frac{ | \Phi|^2 -1}{2} \\ \frac{ | \Phi |^2 +1}{2} \end{pmatrix} + \begin{pmatrix} \n \\ \langle \n ,\Phi \rangle \\ \langle \n , \Phi \rangle \end{pmatrix}.
\end{aligned}
\end{equation}
Which shows that 
\begin{equation}
\label{projectionstereo6}
Y = H^Z \begin{pmatrix} Z_{h} \\ -1 \\ Z_4 \end{pmatrix} + \begin{pmatrix} \n_{h}^Z  \\ 0 \\ \n_4^Z \end{pmatrix}.
\end{equation}

\tocless\subsection{Computations for the conformal Gauss map }
\label{computationsconformalgaussmap}
Let $\Phi  \, : \, \D \rightarrow \R^3$ be a smooth conformal immersion of representation $X$ in $\s^3$  and of conformal Gauss map  $Y$.

Let us first use the expression (\ref{Yphi}).
Then 
$$ \begin{aligned}
Y_z &= H_z  \begin{pmatrix} \Phi \\ \frac{ | \Phi|^2 -1}{2} \\ \frac{ | \Phi |^2 +1}{2} \end{pmatrix} + H \begin{pmatrix}  \Phi_z \\ \langle \Phi_z , \Phi \rangle \\ \langle \Phi_z , \Phi_z \rangle \end{pmatrix} + \begin{pmatrix} \n_z \\ \langle \n_z ,\Phi \rangle \\ \langle \n_z , \Phi \rangle \end{pmatrix}
\end{aligned}
$$
and using (\ref{nzphi})
\begin{equation}
\label{YzPhi}
Y_z=  H_z  \begin{pmatrix} \Phi \\ \frac{ | \Phi|^2 -1}{2} \\ \frac{ | \Phi |^2 +1}{2} \end{pmatrix} - \Omega e^{-2\lambda}  \begin{pmatrix}  \Phi_{\zb} \\ \langle \Phi_{\zb}, \Phi \rangle \\ \langle \Phi_{\zb} , \Phi \rangle \end{pmatrix}.
\end{equation}
Using (\ref{Gauss-CodazziR3}) and (\ref{phizz})  we compute
\begin{equation}
\label{YzzbPhi}
\begin{aligned}
Y_{z \zb} &= H_{z \zb} \begin{pmatrix} \Phi \\ \frac{ | \Phi|^2 -1}{2} \\ \frac{ | \Phi |^2 +1}{2} \end{pmatrix} - \frac{ \left| \Omega \right|^2 }{2}e^{-2\lambda}  \begin{pmatrix}  \n \\ \langle \n, \Phi \rangle \\ \langle \n , \Phi\rangle \end{pmatrix} \\
&= \mathcal{W}( \Phi) \begin{pmatrix} \Phi \\ \frac{ | \Phi|^2 -1}{2} \\ \frac{ | \Phi |^2 +1}{2} \end{pmatrix} - \frac{\left| \Omega \right|^2 e^{-2\lambda} }{2} Y
\end{aligned}
\end{equation}
where 
\begin{equation}
\label{wtorduphi}
\mathcal{W}( \Phi) = H_{z \zb} + \frac{ \left| \Omega \right|^2 e^{-2\lambda} }{2} H \in \R .
\end{equation}

On the other hand 
\begin{equation}
\label{YzzPhi}
\begin{aligned}
Y_{zz} &= H_{zz} \begin{pmatrix} \Phi \\ \frac{ | \Phi|^2 -1}{2} \\ \frac{ | \Phi |^2 +1}{2} \end{pmatrix} + H_z \begin{pmatrix}  \Phi_{z} \\ \langle \Phi_{z}, \Phi \rangle \\ \langle \Phi_{z} , \Phi \rangle \end{pmatrix} - \left( \Omega e^{-2 \lambda} \right)_z \begin{pmatrix}  \Phi_{\zb} \\ \langle \Phi_{\zb}, \Phi \rangle \\ \langle \Phi_{\zb} , \Phi \rangle \end{pmatrix} \\& - \Omega \left(  \frac{H}{2} \begin{pmatrix}  \n \\ \langle \n, \Phi \rangle \\ \langle \n , \Phi\rangle \end{pmatrix} + \frac{1}{2} \begin{pmatrix} 0 \\ 1 \\ 1 \end{pmatrix} \right) 
\end{aligned}
\end{equation}
using (\ref{phizzb}).
Then if we define Bryant's functional as $\mathcal{Q} = \left\langle Y_{zz} , Y_{zz} \right\rangle$ we find

\begin{equation}
\label{QPhi}
\begin{aligned}
\mathcal{Q} &=H_{zz}\Omega - H_z \left( \Omega e^{-2 \lambda} \right)_z e^{2\lambda} + \Omega \frac{H^2}{4} \\
&= \left( \Omega_{\zb} e^{-2\lambda} \right)_z \Omega - \Omega_{\zb} \left( \Omega e^{-2 \lambda} \right)_z + \Omega \frac{H^2}{4}   \text{using (\ref{Gauss-CodazziR3})} \\
&= \left( \Omega_{z\zb} \Omega - \Omega_z \Omega_{\zb} \right) e^{-2\lambda} + \Omega \frac{H^2}{4} \\
&= \Omega^2 e^{-2 \lambda} \left( \frac{\Omega_z}{\Omega} \right)_{\zb} + \Omega \frac{H^2}{4}  =  \Omega^2 e^{-2 \lambda} \left( \frac{\Omega_{\zb}}{\Omega} \right)_{z} + \Omega \frac{H^2}{4}.
\end{aligned}
\end{equation}

\vspace{5mm}

We will now compute using expression (\ref{projectionstereo6}).
Then 
$$ \begin{aligned}
Y_z &= h_z  \begin{pmatrix} X \\ 1 \end{pmatrix} + h \begin{pmatrix}  X_z \\1 \end{pmatrix} + \begin{pmatrix} \vec{N}_z \\ 0 \end{pmatrix}
\end{aligned}
$$
and using (\ref{NzX})
\begin{equation}
\label{YzX}
Y_z=  h_z  \begin{pmatrix} X \\ 1 \end{pmatrix} - \omega e^{-2\Lambda}  \begin{pmatrix}  X_{\zb} \\0 \end{pmatrix}.
\end{equation}
Using (\ref{Gauss-Codazzis3}) and (\ref{Xzz})  we compute
\begin{equation}
\label{YzzbX}
\begin{aligned}
Y_{z \zb} &= h_{z \zb} \begin{pmatrix} X \\ 1 \end{pmatrix} - \frac{ \left| \omega \right|^2 }{2}e^{-2\Lambda}  \begin{pmatrix}  \vec{N} \\ 0 \end{pmatrix} \\
&= \mathcal{W}_{\s^3}( X) \begin{pmatrix}X \\1 \end{pmatrix} - \frac{\left| \omega \right|^2 e^{-2\Lambda} }{2} Y
\end{aligned}
\end{equation}
where 
\begin{equation}
\label{wtorduX}
\mathcal{W}_{\s^3}( X) = h_{z \zb} + \frac{ \left| \omega \right|^2 e^{-2\Lambda} }{2} h \in \R .
\end{equation}
Notice that  using (\ref{projectionstereo3}), (\ref{projectionstereo4}) and (\ref{projectionstereo5})
\begin{equation} \begin{aligned} 
\label{wtorduXPhi}
\mathcal{W}_{\s^3}( X) &= \left( \frac{ |\Phi|^2 +1}{2} H + \langle \n , \Phi \rangle \right)_{z \zb} + \left( \frac{ |\Phi|^2 +1}{2} H + \langle \n , \Phi \rangle \right) \frac{|\Omega|^2e^{-2 \lambda}}{2} \\
&=\left(  \frac{ |\Phi|^2 +1}{2} H_z + \langle \Phi_z , \Phi \rangle H + \langle \n_z , \Phi \rangle \right)_{\zb} +  \frac{ |\Phi|^2 +1}{2} H \frac{ |\Omega|^2 e^{-2\lambda}}{2}\\&+ \langle \n, \Phi \rangle \frac{ |\Omega|^2 e^{-2\lambda}}{2} \\
&= \left(  \frac{ |\Phi|^2 +1}{2} H_z - \Omega e^{-2 \lambda} \langle \Phi_{\zb} , \Phi \rangle \right)_{\zb} +  \frac{ |\Phi|^2 +1}{2} H \frac{ |\Omega|^2 e^{-2\lambda}}{2}\\&+ \langle \n, \Phi \rangle \frac{ |\Omega|^2 e^{-2\lambda}}{2} \\
&=  \frac{ |\Phi|^2 +1}{2} \mathcal{W}(\Phi ) +  \langle \Phi_{\zb} , \Phi \rangle H_z - \Omega_{\zb} e^{-2\lambda} \langle \Phi_{\zb}, \Phi \rangle - \frac{ |\Omega|^2 e^{-2\lambda} }{2} \langle \n, \Phi \rangle\\& + \langle \n, \Phi \rangle \frac{ |\Omega|^2 e^{-2\lambda}}{2} \\
&=  \frac{ |\Phi|^2 +1}{2} \mathcal{W}(\Phi ), 
\end{aligned}\end{equation}
  using (\ref{nzphi}) to obtain the third equality and  (\ref{Gauss-Codazzis3}) to conclude.
On the other hand 
\begin{equation}
\label{YzzX}
\begin{aligned}
Y_{zz} &= h_{zz} \begin{pmatrix} X \\ 1 \end{pmatrix} + h_z \begin{pmatrix} X_{z} \\ 0 \end{pmatrix} - \left( \omega e^{-2 \Lambda} \right)_z \begin{pmatrix}  X_{\zb} \\ 0 \end{pmatrix} - \omega \left(  \frac{h}{2} \begin{pmatrix}  \vec{N} \\ 0 \end{pmatrix} - \frac{1}{2} \begin{pmatrix} X \\ 0 \end{pmatrix} \right) 
\end{aligned}
\end{equation}
using (\ref{phizzb}).
Then if we define $\mathcal{Q} = \left\langle Y_{zz} , Y_{zz} \right\rangle$ we find, once more by applying (\ref{Gauss-Codazzis3}), 
\begin{equation}
\label{QX}
\begin{aligned}
\mathcal{Q} &=h_{zz}\omega - h_z \left( \omega e^{-2 \Lambda} \right)_z e^{2\Lambda} + \omega^2 \frac{h^2+1}{4} \\
&= \left( \omega_{\zb} e^{-2\Lambda} \right)_z \omega - \omega_{\zb} \left( \omega e^{-2 \Lambda} \right)_z + \omega^2 \frac{h^2+1}{4}   \\
&= \left( \omega_{z\zb} \omega - \omega_z \omega_{\zb} \right) e^{-2\Lambda} + \omega^2 \frac{h^2+1}{4} \\
&= \omega^2 e^{-2 \Lambda} \left( \frac{\omega_z}{\omega} \right)_{\zb} + \omega^2 \frac{h^2+1}{4}  =  \omega^2 e^{-2 \Lambda} \left( \frac{\omega_{\zb}}{\omega} \right)_{z} + \omega^2 \frac{h^2+1}{4}.
\end{aligned}
\end{equation}

\tocless\subsection{Formulas in $\s^{4,1}$}
\label{formulas41}
This section is devoted to computations for spacelike immersions in $\s^{4,1}$ without relying on their being the conformal Gauss map of a given immersion.

 Let $Y \, : \, D \rightarrow \s^{4,1}$ be a smooth-spacelike conformal immersion, that is $Y$ satisfies 
$$\left\langle Y_z , Y_z \right\rangle = 0$$
and
$$\left\langle Y_z , Y_{\zb} \right\rangle =: \frac{e^{2 \mathcal{L}}}{2} >0.$$   
Let  $\nu, \, \nu^* \in \mathcal{C}^{4,1}$ such that $e = \left( Y, Y_z ,Y_{\zb}, \nu, \nu^* \right)$ is an orthogonal frame of $\R^{4,1}$, that is 
$$\left\langle Y, \nu \right\rangle = \left\langle Y_z, \nu \right\rangle = \left\langle Y_{\zb}, \nu \right\rangle = \left\langle \nu, \nu \right\rangle =0$$
and
$$\left\langle Y, \nu^* \right\rangle = \left\langle Y_z, \nu^* \right\rangle = \left\langle Y_{\zb}, \nu^* \right\rangle = \left\langle \nu^*, \nu^* \right\rangle =0.$$

We define successively
the tracefree curvature in the direction $\nu$
\begin{equation}
\label{defomeganu}
\Omega_\nu =2 \left\langle Y_{zz}, \nu \right\rangle,
\end{equation}
the tracefree curvature in the direction $\nu^*$
\begin{equation}
\label{defomeganustar}
\Omega_{\nu^*} =2 \left\langle Y_{zz}, \nu^* \right\rangle,
\end{equation}
the mean curvature in the direction $\nu$
\begin{equation}
\label{defHnu}
H_\nu = 2 e^{-2 \mathcal{L} } \left\langle Y_{z \zb}, \nu \right\rangle,
\end{equation}
 and the mean curvature in the direction $\nu^*$
\begin{equation}
\label{defHnustar}
H_{\nu^*} = 2 e^{-2 \mathcal{L} } \left\langle Y_{zz}, \nu^* \right\rangle.
\end{equation}

Then 
\begin{equation}
\label{Yzz}
Y_{zz} = 2 \mathcal{L}_z Y_z + \frac{  \Omega_\nu}{2 \langle \nu , \nu^* \rangle } \nu^* +  \frac{  \Omega_{\nu^*}}{2 \langle \nu , \nu^* \rangle }  \nu,
\end{equation}
and
\begin{equation}
\label{Yzzb}
Y_{z\zb}= \frac{H_\nu e^{2 \mathcal{L} } }{2 \langle \nu , \nu^* \rangle } \nu^* + \frac{H_{\nu^*}e^{2 \mathcal{L} }}{2 \langle \nu , \nu^* \rangle }  \nu  - \frac{e^{ 2 \mathcal{L} }}{2} Y.
\end{equation}

Further 
\begin{equation} \label{nuzY} \left\langle \nu_z, Y \right\rangle   = \left( \left\langle \nu , Y \right\rangle \right)_z -  \left\langle \nu , Y_z \right\rangle = 0, \end{equation} and with  (\ref{Yzz}),
\begin{equation} \label{nuzYz} \begin{aligned}
\left\langle \nu_z, Y_z\right\rangle &=  \left( \left\langle \nu , Y_z \right\rangle \right)_z -  \left\langle \nu , Y_{zz} \right\rangle \\
&=  -2 \mathcal{L}_z \left\langle \nu , Y_z \right\rangle - \frac{  \Omega_\nu}{2 \langle \nu , \nu^* \rangle }  \langle \nu , \nu^* \rangle - \frac{  \Omega_\nu^*}{2 \langle \nu , \nu^* \rangle }  \langle \nu , \nu \rangle \\
&= - \frac{ \Omega_\nu}{2},
\end{aligned} \end{equation} while with (\ref{Yzzb}),
\begin{equation} \label{nuzYzb} \begin{aligned}
\left\langle \nu_z, Y_{\zb}\right\rangle &=  \left( \left\langle \nu , Y_{\zb} \right\rangle \right)_z -  \left\langle \nu , Y_{z \zb} \right\rangle \\
&= -\frac{  H_\nu e^{2 \mathcal{L} } }{2 \langle \nu , \nu^* \rangle }  \langle \nu , \nu^* \rangle  \\
&= - \frac{ H_\nu e^{2 \mathcal{L} }}{2},
\end{aligned} \end{equation}
and
$$
\left\langle \nu_z, \nu \right\rangle = \left( \langle \nu , \nu \rangle \right)_z - \langle \nu , \nu_z \rangle, $$
meaning 
\begin{equation}
\label{nuznu} 
\left\langle \nu_z, \nu \right\rangle = 0.
\end{equation}

Combining (\ref{nuzY}), (\ref{nuzYz}), (\ref{nuzYzb}) and (\ref{nuznu}) yields
\begin{equation}
\label{nuzformule}
\nu_z=  - \left\langle \nu_z, \nu^* \right\rangle \nu - H_\nu Y_z - \Omega_\nu e^{-2 \mathcal{L} } Y_{\zb}.
\end{equation}
Similarly
\begin{equation} \label{nustarzY} \left\langle \nu^*_z, Y \right\rangle   = \left( \left\langle \nu^* , Y \right\rangle \right)_z -  \left\langle \nu^* , Y_z \right\rangle = 0, \end{equation} and with (\ref{Yzz}), 
\begin{equation} \label{nustarzYz} \begin{aligned}
\left\langle \nu^*_z, Y_z\right\rangle &=  \left( \left\langle \nu^* , Y_z \right\rangle \right)_z -  \left\langle \nu^* , Y_{zz} \right\rangle \\
&=  -2 \mathcal{L}_z \left\langle \nu^* , Y_z \right\rangle - \frac{  \Omega_{\nu^*}}{2 \langle \nu , \nu^* \rangle }  \langle \nu , \nu^* \rangle - \frac{  \Omega_\nu}{2 \langle \nu , \nu^* \rangle }  \langle \nu^* , \nu^* \rangle    \\
&= - \frac{ \Omega_{\nu^*}}{2},
\end{aligned} \end{equation} while with  (\ref{Yzzb})
\begin{equation} \label{nustarzYzb} \begin{aligned}
\left\langle \nu^*_z, Y_{\zb}\right\rangle &=  \left( \left\langle \nu^* , Y_{\zb} \right\rangle \right)_z -  \left\langle \nu^* , Y_{z \zb} \right\rangle \\
&= -\frac{  H_{\nu^*} e^{2 \mathcal{L} } }{2 \langle \nu , \nu^* \rangle }  \langle \nu , \nu^* \rangle   \\
&= - \frac{ H_{\nu^*} e^{2 \mathcal{L} }}{2},
\end{aligned} \end{equation}
\begin{equation}
\label{nustarznustar}
\langle \nu^*_z \nu^* \rangle = 0,
\end{equation}
\begin{equation}
\label{nustarzformule}
\nu^*_z=  - \left\langle \nu^*_z, \nu \right\rangle \nu^* - H_{\nu^*} Y_z - \Omega_{\nu^*} e^{-2 \mathcal{L} } Y_{\zb}.
\end{equation}
Then 
\begin{equation}
\label{nunustar}
\begin{aligned}
\left\langle \nu_z, \nu_z \right\rangle &= H_\nu \Omega_\nu \\
\left\langle \nu^*_z, \nu^*_z \right\rangle &= H_{\nu^*} \Omega_{\nu^*}.
\end{aligned}
\end{equation}

\addcontentsline{toc}{section}{Bibliography}
\bibliographystyle{plain}
\bibliography{bibliography}

\end{document}